\newif\ifpreprint
\newif\ifipco
\newif\ifdoubleblind
\newif\ifsiam

\preprinttrue 

\ifpreprint
\documentclass[12pt]{article}

\usepackage{fullpage}

\usepackage{amsthm}
\usepackage{multirow}

\usepackage{csvsimple}	
\usepackage{booktabs}   
\AtBeginEnvironment{tabular}{\small}
\usepackage{adjustbox}  

\makeatletter
\renewcommand{\paragraph}[1]{\textbf{#1.}\;}
\makeatother

\fi 

\ifipco
\documentclass[runningheads]{llncs}
\usepackage[T1]{fontenc}

\fi 

\ifsiam
\documentclass[final,hidelinks,onefignum,onetabnum]{siamart250106}
\usepackage{csvsimple}	
\usepackage{booktabs}   
\usepackage{adjustbox}  

\headers{Verification of First-Order Methods via MILP}{V. Ranjan, J. Park, S. Gualandi, A. Lodi, and B. Stellato}

\author{Vinit Ranjan\thanks{Department of Operations Research and Financial Engineering, Princeton University, New Jersey, USA
  (\email{vranjan@princeton.edu}, \email{jisunpark@princeton.edu},
  \email{bstellato@princeton.edu}).\\
  V. Ranjan and B. Stellato are supported by the NSF CAREER Award ECCS-2239771. J. Park (affiliated with the Research Institute of Mathematics, Seoul National University) is supported by the National Research Foundation of Korea (NRF) grant funded by the Korea government (MSIT) (RS-2024-00353014). B. Stellato and J. Park are also supported by the ONR YIP Award N000142512147.}
  \and Jisun Park\footnotemark[2]
  \and Stefano Gualandi\thanks{Department of Mathematics, University of Pavia, Pavia, Italy (\email{stefano.gualandi@unipv.it}).\\
  S. Gualandi acknowledges the contribution of the National Recovery and Resilience Plan, Mission 4 Component 2 - Investment 1.4 - National Center for HPC, Big Data and Quantum Computing (project code: CN\_00000013) - funded by the European Union - NextGenerationEU.}
  \and Andrea Lodi\thanks{Jacobs Technion-Cornell Institute, Cornell Tech, New York, USA (\email{andrea.lodi@cornell.edu}).}
  \and Bartolomeo Stellato\footnotemark[2]
}
\fi 

\usepackage{graphicx}
\usepackage{pgfplots}
\pgfplotsset{compat=1.17} 
\usepackage{tikz}
\usepackage{xcolor}
\definecolor{myblue}{RGB}{72, 145, 220}
\usepackage[font=small,
skip=5pt,
belowskip=0pt
]{caption} 

\usepackage{bm}  
\usepackage{stmaryrd}
\usepackage{amsmath,amssymb,enumitem,mathtools}

\ifsiam
\else
\usepackage[linesnumbered,ruled,vlined]{algorithm2e}

\usepackage[colorlinks=true, citecolor=cyan, linkcolor=cyan, urlcolor=cyan]{hyperref}
\fi 

\newcommand{\bnote}[1]{}
\newcommand{\vnote}[1]{}
\newcommand{\snote}[1]{}
\newcommand{\anote}[1]{}
\newcommand{\jnote}[1]{}

\newcommand{\errata}[1]{{\color{black}#1}}

\newcommand{\eg}{{\it e.g.}}
\newcommand{\ie}{{\it i.e.}}

\newcommand{\ubar}[1]{\text{\b{$#1$}}}

\newcommand{\ones}{\mathbf 1}
\newcommand{\reals}{{\mbox{\bf R}}}

\newcommand{\symm}{{\mbox{\bf S}}}  

\newcommand{\tpose}{T}

\newcommand{\diag}{\mathop{\bf diag}}
\newcommand{\prox}{{\bf prox}}
\newcommand{\conv}{{\bf conv}}

\newcommand{\argmin}{\mathop{\rm argmin}}

\newcommand{\ReLU}{\mathop{\bf ReLU}}

\newcommand{\normalcone}{{\mathop{N}}}

\ifpreprint
\newtheorem{theorem}{Theorem}[section]  
\newtheorem{lemma}{Lemma}[section]  
\newtheorem{proposition}{Proposition}[section]

\theoremstyle{definition}

\theoremstyle{remark}

\newtheorem{remark}{Remark}[section]

\fi

\ifpreprint
\fi 
\ifipco
\fi

\ifpreprint
\title{Exact Verification of First-Order Methods\\ via Mixed-Integer Linear Programming}
\fi

\ifsiam
\title{Exact Verification of First-Order Methods\\ via Mixed-Integer Linear Programming\thanks{Submitted to the editors April 5, 2025.}}
\fi

\ifipco
\titlerunning{Exact Verification of First-Order Methods using MILP}

\ifdoubleblind
\author{}
\authorrunning{}
\institute{}

\else

\author{Vinit Ranjan\inst{1} \and Jisun Park\inst{1} \and Stefano Gualandi\inst{2} \and Andrea Lodi\inst{3} \and Bartolomeo Stellato\inst{1}}
\authorrunning{V.\ Ranjan et al.}

\institute{Princeton University, Princeton, NJ 08544, USA \email{\{vranjan, jisunpark, bstellato\}@princeton.edu} \and
University of Pavia, 27100 Pavia, Italy\\\email{stefano.gualandi@unipv.it}\\
\and Cornell Tech, New York, NY 10044, USA\\
\email{andrea.lodi@cornell.edu}}
\fi

\fi 

\ifpreprint
\author{Vinit Ranjan, Jisun Park, Stefano Gualandi, Andrea Lodi, and Bartolomeo Stellato}

\fi 

\begin{document}

\maketitle

\begin{abstract}
We present exact mixed-integer linear programming formulations for verifying the performance of first-order methods for parametric quadratic optimization.
We formulate the verification problem as a mixed-integer linear program where the objective is to maximize the infinity norm of the fixed-point residual after a given number of iterations. 
Our approach captures a wide range of gradient, projection, proximal iterations through affine or piecewise affine constraints.
We derive tight polyhedral convex hull formulations of the constraints representing the algorithm iterations. 
To improve the scalability, we develop a custom bound tightening technique combining interval propagation, operator theory, and optimization-based bound tightening.
Numerical examples, including linear and quadratic programs from network optimization, sparse coding using Lasso, and optimal control, show that our method provides several orders of magnitude reductions in the worst-case fixed-point residuals, closely matching the true worst-case performance.
\end{abstract}


\section{Introduction}\label{sec:intro}
Convex Quadratic Programming~(QP) is a fundamental class of optimization problems in applied mathematics, operations research, engineering, and computer science~\cite{boydConvexOptimization2004} with several applications from finance~\cite{boydMultiPeriodTradingConvex2017,markowitz} and machine learning~\cite{candes2007lasso,svm,huber1964huber,lasso}, to signal processing~\cite{chenAtomicDecompositionBasis2001,sig_proc} and model predictive control~\cite{borrelliPredictiveControlLinear2017}. 
A particularly important subclass of QP is Linear Programming~(LP)~\cite{dantzigLinearProgrammingExtensions1998}, which is widely used in scheduling~\cite{hanssmannLinearProgrammingApproach1960}, network optimization~\cite{bertsekas1991linear}, chip design~\cite{kimLinearProgrammingbasedAlgorithm2003}, and resource allocation~\cite{blandAllocationResourcesLinear1981}. 
In many such applications, solving the optimization problem is part of a wider system, imposing strict computational requirements on the algorithm used.
In particular, these algorithms consist of repeated iterations, generating a sequence of points that converge to an optimal solution, if the problem is solvable.
Depending on the operational speed of each application, the time available to solve each optimization problem may vary from minutes to milliseconds, which, in turn, sets an exact limit on the number of allowed iterations.
Therefore, to safely deploy optimization algorithms, we must ensure that they \emph{always converge} within the prescribed iteration limit, \ie, in real-time.

First-order optimization algorithms have attracted an increasing interest over the last decade due to their low per-iteration cost and memory requirements, and their warm-start capabilities.
These properties are particularly useful, on the one hand, in large-scale optimization problems arising in data-science and machine learning, and, on the other hand, in embedded optimization problems arising in engineering and optimal control.
First-order methods date back to the 1950s and were first applied to QPs for the development of the Frank-Wolfe algorithm~\cite{MargueriteWolfe1956}.
Recently, several general-purpose solvers based on first-order methods have appeared, including PDLP~\cite{applegatePracticalLargeScaleLinear2021,applegateFasterFirstorderPrimaldual2023} for LP, OSQP~\cite{stellatoOSQPOperatorSplitting2020} and QPALM~\cite{hermansQPALMProximalAugmented2022} for QP, and SCS~\cite{odonoghue:21,ocpb:16} and COSMO~\cite{garstkaCOSMOConicOperator2021} for conic optimization.
These methods feature several algorithmic enhancements that make them generally reliable, including acceleration~\cite{odonoghue:21}, and restarts~\cite{applegatePracticalLargeScaleLinear2021,applegateFasterFirstorderPrimaldual2023}, which guarantee worst-case linear convergence rates.
However, their exact non-asymptotic empirical behavior for a fixed budget of iterations is not clearly understood.
In particular, their practical convergence behavior can be very different, and sometimes significantly better, than the worst-case bounds that do not account for any application-specific problem structure and warm-starts used.

In this paper, we introduce an exact and scalable Mixed-Integer Linear Programming (MILP) formulation of the problem of analyzing the performance of first-order methods in LP and QP.

\subsection{The problem}
We are interested in solving QPs of the form
\begin{equation}\label{eq:prob}
    \begin{array}{ll}
        \text{minimize} & (1/2)z^TPz + q(x)^Tz\\
        \text{subject to} & Az + r = b(x)\\
        & r \in C_1 \times C_2 \times \dots \times C_L,
    \end{array}
\end{equation}
with variables $z \in \reals^n$ and $r \in \reals^m$, and instance parameters $x \in X \subseteq \reals^p$, where $X$ is a polyhedron.
The sets $C_i \subseteq \reals^{n_i}$ come from a selected list:  the zero set $\{0\}^{n_i}$ to model equality constraints, the nonnegative orthant $\reals_+^{n_i}$ to model one-sided inequality constraints, and the box $[\ell, u] \subseteq \reals^{n_i}$ to model double-sided inequality constraints.
We denote the objective as $f(z, x) = (1/2)z^TPz + q(x)^Tz$ with positive semidefinite matrix~$P \in \symm_{+}^n$ and affine function $q(x) \in \reals^n$.
The constraints are defined by matrix $A \in \reals^{m \times n}$, and affine function $b(x) \in \reals^m$.
If $P=0$, problem~\eqref{eq:prob} becomes an LP.
We focus on settings where we solve several instances of problem~\eqref{eq:prob} with varying instance parameters $x$. We assume problem~\eqref{eq:prob} to be solvable for any $x \in X$ and we denote the optimal solutions as~$z^\star(x)$.
Problems where $q(x)$ and $b(x)$ change arise in a wide range of settings, including sparse coding using Lasso~\cite{lasso} ($x$ being the input signal), model predictive control~\cite{borrelliPredictiveControlLinear2017} ($x$ being the state of a linear dynamical system), and network flow problems~\cite[Ch.\ 1]{bertsekas1991linear} ($x$ being the node supplies).
Our goal is to solve problem~\eqref{eq:prob} using \emph{first-order methods} of the form
\begin{equation}
s^{k+1} = T(s^k, x),\quad k=0,1,\dots,
\end{equation}
where $T$ is a \emph{fixed-point operator} that maps the (primal-dual) iterate $s^k \in \reals^d$ to $s^{k+1} \in \reals^d$ depending on the instance parameter $x$.
\errata{For prescribed iteration-dependent schedules, we also allow $s^{k+1}=T_k(s^k,x)$ and interpret $T^K$ below as $T_{K-1}\circ\cdots\circ T_0$ at fixed $x$.}
For any convergent first-order method and instance $x$, $T$ can be constructed such that the fixed-points of $T$ correspond to the optimal solution of the given instance.
We assume the operator to be $L$-Lipschitz in its first argument with $L>0$ and, therefore, single-valued~\cite[Sec.\ 2.1]{ryuLargeScaleConvexOptimization2022}.
Operators of this form encode a wide range of algorithms, including gradient descent, proximal gradient descent, and operator splitting algorithms~\cite{ryuLargeScaleConvexOptimization2022}. 
We denote the initial iterate set as the polyhedron $S \subset \reals^d$.
For cold-started algorithms, we have~${S = \{0\}}$.
For a given $x$, \emph{fixed-points} $s^\star = T(s^\star, x)$ correspond to the (primal-dual) optimal solutions of problem~\eqref{eq:prob} (\ie, we can compute $z^\star$ from $s^\star$).
We define the \emph{fixed-point residual} at iterate $s^k$ for parameter $x$ as~$T(s^k, x) - s^k$.
As $k \to \infty$, the fixed-point residual converges to 0 for common cases of operators (\eg, averaged or contractive~\cite[Corollary 5.8, Theorem 1.50]{bauschkeConvexAnalysisMonotone2017}, \cite[Theorem 1]{ryuLargeScaleConvexOptimization2022}).
Furthermore, it is often the case that the fixed-point residual with proper scaling can model the optimality conditions of~\eqref{eq:prob}.
In Section~\ref{sec:experiments}, we provide algorithm-specific scaling matrices $H$ such that $H (s^K - s^{K-1})$ recovers the optimality conditions of the specific problem.
Therefore, the fixed-point residual is an easily computable measure of suboptimality and infeasibility.
It is worth noting that in some cases, the fixed-point residual may not monotonically decrease; it may even increase for some iterations.
However, for any convergent first-order method, the residual will still converge to 0 as $k \to \infty$.

We are interested in numerically verifying that the worst-case magnitude of the fixed-point residual after $K$ iterations is \errata{at most} an accepted tolerance $\epsilon > 0$.
Specifically, we would like to verify the following condition:
\begin{equation}\label{eq:verify}
    \|H ( T^K(s^0, x) - T^{K-1}(s^0, x) ) \|_{\infty} \le \epsilon,\quad \forall x \in X, \;s^0 \in S,
\end{equation}
where $T^K(s^0, x)$ denotes the $K$-times composition of $T$ with parameter $x$, starting from iterate $s^0$.
Note that the $\ell_\infty$-norm is a standard way to measure convergence and is commonly used in most general-purpose solvers~\cite{applegatePracticalLargeScaleLinear2021,gurobi,ocpb:16,stellatoOSQPOperatorSplitting2020}.
Verifying condition~\eqref{eq:verify} boils down to checking if the optimal value of the optimization problem
\begin{equation}\label{eq:vp}\tag{${\rm VP}$}
    \delta_K = \begin{array}[t]{ll}
        \text{maximize} & \|H (s^{K} - s^{K-1})\|_{\infty}\\ 
        \text{subject to} & s^{k+1} = T(s^k, x),\quad k=0,1\dots, K-1,\\
                          & s^0 \in S,\quad x \in X,
    \end{array}
\end{equation}
is \errata{at most} $\epsilon$, where the variables are the problem parameters $x$ and the iterates $s^0,\dots,s^K$.
In the rest of the paper, we will describe scalable ways to formulate and solve this problem as a MILP.
We refer to~\eqref{eq:vp} as the \emph{verification problem}.

\subsection{Related works}\label{sec:relatedworks}
There has been much interest in both analyzing theoretical convergence properties of first-order methods and verifying their practical properties in different applications.

\paragraph{Performance estimation}
The optimization community has recently devoted significant attention on computer-assisted tools to analyze the worst-case finite step performance of first-order optimization algorithms.
Notable examples include the performance estimation problem (PEP)~\cite{droriPerformanceFirstorderMethods2014,taylorExactWorstCasePerformance2017,taylorSmoothStronglyConvex2017} approach that formulates the worst-case analysis as a semidefinite program (SDP), and the integral quadratic constraints (IQCs) framework that models optimization algorithms as dynamical systems, analyzing their performance using control theory~\cite{lessardAnalysisDesignOptimization2016}.
Interestingly, solutions to IQCs can be viewed as feasible solutions of a specific PEP formulation searching for optimal linear convergence rate using Lyapunov functions~\cite{taylorLyapunovFunctionsFirstOrder2018}.
By defining the appropriate constraints (\ie, interpolation conditions), the PEP framework can model complex methods, including operator splitting~\cite{collaAutomaticPerformanceEstimation2023,ryuOperatorSplittingPerformance2020,zamaniExactWorstcaseConvergence2024}, linear operators~\cite{bousselmiInterpolationConditionsLinear2024}, and stochastic optimization algorithms~\cite{taylorStochasticFirstorderMethods2019}.
Thanks to user-friendly toolboxes such as PEPit~\cite{goujaudPEPitComputerassistedWorstcase2024}, PEP has promoted numerous discoveries, including the design of optimal algorithms~\cite{droriEfficientFirstorderMethods2020}, by constructing custom inequality constraints~\cite{parkOptimalFirstOrderAlgorithms2024} or by formulating the design problem as bilevel optimization where PEP serves as the subproblem~\cite{dasguptaBranchandboundPerformanceEstimation2024}.
We provide a detailed discussion of our work and how it differs from the PEP framework in Section~\ref{sec:pep_comparison}.


\paragraph{Verification for real-time optimization}
In many embedded optimization problems from engineering or optimal control, it is important to guarantee an approximately accurate solution in real-time, which corresponds to a tight iteration limit \cite{jerezEmbeddedOnlineOptimization2014,odonoghueSplittingMethodOptimal2013,stathopoulosOperatorSplittingMethods2016}.
This motivates our verification problem as the convergence behavior in the short-term (\ie~for a small, fixed number of iterations) can behave very differently from the long-term convergence rate~\cite[Section VII]{richterComputationalComplexityCertification2012}\cite{stellatoOSQPOperatorSplitting2020}.
Due to their widespread applicability, many works have focused on convergence verification and complexity estimates for convex QPs~\cite{arnstromExactComplexityCertification2020, arnstromComplexityCertificationProximalPoint2021, patrinosDouglasrachfordSplittingComplexity2014}.
Critically, these methods use the parametric QP structure by looking at box constraints~\cite{ciminiComplexityConvergenceCertification2019} or the active set of constraints at optimality~\cite{ciminiExactComplexityCertification2017}.
In optimal control, model predictive control problems are of particular interest as the individual parametric QPs are often not large problems, but it is imperative that they can be approximately solved in a small amount of time~\cite{bemporadSimpleCertifiableQuadratic2012, ciminiEmbeddedModelPredictive2021}.
Model predictive control is particularly important among the examples we consider in this paper, as it satisfies all properties we desire: it gives a small to medium family of parametric convex QPs for which we aim to verify the performance of first-order methods within a small number of iterations.

\paragraph{Neural network verification}
Our approach closely connects with the neural network verification literature. The main goal in this area is to verify that, for every allowable input (\eg, all permissible pixel perturbations of training images), the output of a trained neural network satisfies certain properties (\eg, correctly classifies the images). This ensures that there are \emph{no adversarial examples} that are misclassified.
Similarly, in this paper, we verify that at a given iteration, for every allowable problem instance, the norm of the residual remains below a precision threshold~$\epsilon$, guaranteeing that there are \emph{no adversarial problem instances} that are not solved to the desired precision. 
The neural network verification problem is inherently challenging and nonconvex, and various approaches frame it as a MILP~\cite{JMLR:v23:22-0277,fischettiDeepNeuralNetworks2018,tjeng2018evaluating} by encoding ReLU activation functions with discrete variables. Lifted representations of these activation functions can yield tighter problem relaxations and have been shown to reduce solution times~\cite{andersonStrongMixedintegerProgramming2020,tjandraatmadjaConvexRelaxationBarrier2020,tsayPartitionBasedFormulationsMixedInteger2021}.
However, successfully verifying large, general neural network architectures requires specialized branch-and-bound algorithms~\cite{bunelUnifiedViewPiecewise2018,shiNeuralNetworkVerification2024} featuring specialized bound tightening~\cite{hojnyVerifyingMessagepassingNeural2024}, GPU-based bound propagation~\cite{wangBetaCROWNEfficientBound2021} and custom cutting planes~\cite{zhangGeneralCuttingPlanes2022}. Despite recent progress, the complex structure of trained neural networks still makes it challenging to verify architectures with tens of layers. 
While also relying on MILP formulations, our work exploits the specific convergence behavior of first-order methods to verify algorithm performance over tens of iterations.
In addition, 
we model nonconvex iteration functions (\eg, soft thresholding). 


\subsection{Relation to Performance Estimation Problem (PEP)}\label{sec:pep_comparison}
It is important to highlight some key differences between our proposed framework and the existing PEP framework.
They are complementary tools since, at their core, both frameworks aim to answer a similar question: \emph{what is the worst-case performance of a given first-order method on some family of optimization problems?}
In PEP, the answer is provided as a solution to the following optimization problem
\begin{equation}
\label{prob:perf_est}
	\begin{array}{lll}
		\text{maximize} & \mathcal{P}(f, \{s^k\}) & \text{(Performance metric)} \\
        \text{subject to} &
        f \in \mathcal{F} & \text{(Optimization problem)} \\
        & \{s^k\} = \mathcal{A}(s^0, f), \quad s^0 \in S & \text{(First-order method)}
	\end{array}
\end{equation}
where $\mathcal{P}$ is a performance metric, $\mathcal{F}$ is a family of optimization problem represented by a function or operator $f$, $S$ is a set where the initial iterate $s^0$ resides, and $\mathcal{A}$ is a given first-order method. The solution $(f^\star, s^{0\star})$ of~\eqref{prob:perf_est} encodes the problem instance that achieves the worst-case performance.
The key difference between the PEP framework and this paper comes from the encodings of $f\in\mathcal{F}$ and $s^0 \in S$.

The PEP framework represents $f \in \mathcal{F}$ by imposing interpolation conditions on algorithm iterates $\{s^k\} \subseteq \reals^d$.
Since such conditions are quadratic, the above problem becomes a non-convex Quadratically-constrained Quadratic Program (QCQP), and we denote this problem as (QCQP-PEP).
Notably, if the formulation depends only on the inner products of~$\{s^k\}$, then we can exactly reformulate QCQP-PEP as a tractable convex SDP (SDP-PEP)~\cite[Section 3.2]{droriPerformanceFirstorderMethods2014}, \cite[Theorem 5]{taylorSmoothStronglyConvex2017}.
The SDP-PEP enjoys favorable properties such as dimension independence and provides direct convergence proofs by analyzing its dual~\cite[Section 3.2]{droriPerformanceFirstorderMethods2014}, \cite[Section 3.4]{taylorSmoothStronglyConvex2017}.
However, applying the SDP-PEP formulation to~\eqref{eq:prob} can result in a pessimistic, \ie, non-tight, worst-case result.
First of all, the problem class $\mathcal{F}$ of~\eqref{eq:prob} cannot be represented in terms of the interpolation conditions of SDP-PEP.
This is due to the fact that the worst-case instance of SDP-PEP is invariant under orthogonal transformation, whereas it is not in our setup as observed in prior work~\cite[Section 3.1]{ranjanVerificationFirstOrderMethods2025}.
Also, SDP-PEP only allows for $S$ that can be represented by inner product values of iterates such as $S = \{s^0 \mid \|s^0-s^\star\|_2 \le R\}$.
Therefore, initialization schemes, such as warm-starting techniques, where $S$ depends on constant vectors in $\reals^d$, cannot be encoded by SDP-PEP.
In the same context, it is not possible in SDP-PEP to represent the performance metric of termination criteria using the $\ell_\infty$-norm because this norm cannot be represented in terms of the inner products of the iterates~$\{s^k\}$.
In contrast to the SDP-PEP formulation, our proposed verification problem (VP) can successfully mitigate the aforementioned pessimism by explicitly representing parametric QPs as $f$, as well as~$\ell_\infty$-norms and warm-starts using MILP.

We would like to clarify that, even though the SDP-PEP limitations mentioned earlier do not apply to QCQP-PEP, solving the latter can be numerically challenging, as non-convex QCQPs are NP-hard~\cite{wangTightnessSDPRelaxations2021}.
Furthermore, algorithms for non-convex QCQPs, such as spatial branch-and-bound, are underdeveloped compared to algorithms and technology for MILP solvers~\cite{gurobi}.
We do not present this work as an intended replacement for PEP; instead, we answer the same question with an approach that is neither an extension nor a special case to the prior PEP-based framework.

\subsection{Our contributions}
In this paper, we introduce a new MILP framework to exactly verify the performance of first-order optimization algorithms for QPs. 
\begin{itemize}
    \item We formulate the performance verification problem as a MILP where the objective is to maximize the $\ell_\infty$-norm of the convergence residual after a given number of iterations.
    We model common gradient, projection, and proximal steps as either affine or piecewise affine constraints in the verification problem. 
    \item We construct convex hull formulations of piecewise affine steps of the form of the soft-thresholding operator, the rectified linear unit, and the saturated linear unit, and show that the corresponding separation problems can be solved in polynomial time.
    \item We introduce various techniques to enhance the scalability of the MILP. Specifically, we develop a custom bound tightening technique combining interval propagation, convergence properties from operator theory, and optimiza\-tion-based bound tightening. Furthermore, we devise a sequential technique to solve the verification problem by iteratively increasing the number of steps until we reach the desired number of iterations. 
\item We compare our method against state-of-the-art computer assisted performance estimation techniques, on LP instances in network optimization and QP instances in sparse coding using Lasso as well as optimal control, obtaining multiple orders of magnitude reduction on worst-case norm of the scaled fixed-point residual.
\end{itemize}

\section{Mixed-Integer Linear Programming formulation}\label{sec:milp_formulation}
The remainder of this paper focuses on modeling~\eqref{eq:vp} using MILP techniques along with our many problem-specific improvements to speed up the solving procedure.
\errata{For this construction, we assume finite valid component bounds on all intermediate variables over the prescribed horizon.}
    For a general overview of MILP, we refer to~\cite{confortiIntegerProgramming2014}.
Since algorithms often involve multiple steps for each iteration, we use the following notation to represent each iteration $k$:
\begin{equation}\label{eq:iterations}
    s^{k+1} = T(s^k, x)
    \quad \iff \quad \begin{array}{ll}
	v^{1} &= \phi_1(s^k, x)\\
	v^{2} &= \phi_2(v^1, x)\\
	 &\;\vdots\\
	s^{k+1} &= \phi_l(v^{l-1}, x),\\
    \end{array}
\end{equation}
where $\phi_{1},\dots,\phi_{l}$ are the intermediate step mappings.
In the intermediate steps, proximal algorithms rely on evaluating the proximal operator of a function $g$, defined as~$\prox_{g}(v) = \argmin_y\;(g(y) + (1/2)\|y - v\|_2^2)$~\cite{beckFirstOrderMethodsOptimization2017,parikhProximalAlgorithms2014}. 
Most intermediate steps arising in LP and QP can be computed as either \emph{affine} or \emph{piecewise affine} steps~\cite[Sec.\ 6]{parikhProximalAlgorithms2014}, which we express as MILP constraints.

\subsection{Objective formulation}\label{subsec:obj}
To model the $\ell_\infty$-norm in the objective of problem~\eqref{eq:vp}, we define~$\delta_K = \|H(s^K - s^{K-1})\|_{\infty} = \|t\|_{\infty}$, with $t = H(s^{K} - s^{K-1}) \in \reals^{h}$ being the scaled residual at the last step.
Using the lower bounds $\ubar{s}^{K-1}, \ubar{s}^{K}$ and the upper bounds, $\bar{s}^{K-1}$ and $\bar{s}^{K}$, of the last two iterates $s^{K-1}$ and $s^{K}$, we can derive the lower and upper bounds on the residual $t$, denoted $\ubar{t}$ and $\bar{t}$, respectively~\cite[Sec. 3]{gowalScalableVerifiedTraining2019}.
We provide the explicit bound formulas in Section~\ref{subsec:bound_tightening}.

\errata{Let $M_t = \max_{1\le i\le h}\max\{\bar{t}_i,-\ubar{t}_i\}$, where $h\ge1$.}
Then, by writing $t$ as the difference of two nonnegative variables $t^{+}$ and $t^{-}$, \ie, $t = t^{+} - t^{-}$, we can represent the absolute value of its components as $t^{+} + t^{-}$ and write~$\delta_{K}$ as the following constraints:
\begin{equation}\label{eq:obj_reformulation}
    \begin{array}{ll}
    &t = t^{+} - t^{-},\quad t^{+} \le \bar{t} \odot \omega, \quad t^{-} \le -\ubar{t} \odot (\ones - \omega)\\
    & t^{+} + t^{-} \le \delta_K \ones \le t^{+} + t^{-} + \errata{M_t (\ones - \gamma)}\\
    & \ones^T\gamma = 1,\quad t^{+} \ge 0,\quad  t^{-} \ge 0,
    \end{array}
\end{equation}
where $\odot$ is the elementwise product. 
Here, we introduced variable $\omega \in \{0, 1\}^{h}$ to represent the absolute values of the components of $t$, and variable $\gamma \in \{0, 1\}^{h}$ to represent the maximum inside the $\ell_\infty$-norm.
We remark that the lower bound constraints $t^{+} + t^{-} \le \delta_K \ones$ provide a complete description of $\delta_K$, but they are not strictly necessary due to the maximization.



\subsection{Affine steps}
Consider affine steps of the following form:
\begin{equation}\label{eq:affinestep}
    \phi(v, x) = \{w \in \reals^d \mid M w = By,\quad y = (v, x)\},
\end{equation}
where $M \in \reals^{d \times d}$ is an invertible matrix and $B \in \reals^{d \times (d + p)}$.
We directly express these iterations as linear equality constraints for the verification problem~\eqref{eq:vp}.

\subsubsection{Explicit steps}
When $M=I$, affine steps represent explicit iterations.

\paragraph{Gradient step}
A gradient step of size $\eta > 0$ for a quadratic function $f(v,x) = (1/2)v^TPv + x^Tv$ can be written as
\begin{equation}\label{eq:gd_step}
    \phi(v, x) = \{w \in \reals^d \mid w = v - \eta \nabla f(v, x)  = (I - \eta P)v - \eta x\},
\end{equation}
where $\nabla f(v, x) = Pv + x$ is the gradient with respect to the first argument of~$f$.
Therefore, it is an affine step with $M=I$ and $B = [I\! -\!\eta P\;  -\!\eta I]$.

\paragraph{Momentum step}
For a momentum step~\cite{nesterov1983method}\cite[Sec.\ 2.2]{ranjanVerificationFirstOrderMethods2025}, an operator is applied to a linear combination of iterates at both iteration $k$ and iteration $k-1$.
Let $g$ denote the intermediate step variables that momentum is applied to, and $v = (g^k, g^{k-1})$ is the concatenation of intermediate step variables from iteration $k$ and $k-1$.
Given $\beta^k > 0$, momentum updates are defined by:
\begin{align*}
    \tilde{v} &= g^k + \beta^k(g^k - g^{k-1})\\
    w &= \tilde{v} - \eta\nabla f(\tilde{v}, x),
\end{align*}
and so can be written in our framework via substitution as:
\begin{equation*}
    \phi(v, x) = \{w \in \reals^d \mid w = (1+\beta^k)(I\! -\!\eta P)g^k -\!\beta^k(I\! -\!\eta P)g^{k-1} - \eta x \}.
\end{equation*}
Therefore, it is an affine step with $M=I$ and $B = [(1+\beta^k)(I\! -\!\eta P)\;  -\!\beta^k(I\! -\!\eta P) \; -\!\eta I]$.

\paragraph{Shrinkage operator}
The shrinkage operator represents the proximal operator of the squared $\ell_2$-norm function $f(v) = (1/2)\|v \|^2$~\cite[Sec.\ 6.1.1]{parikhProximalAlgorithms2014}.
Given $\lambda > 0$, we have
\begin{equation*}
    \phi(v, x) = \prox_{\lambda f}(v) = \{ w\in \reals^d \mid w = (1+\lambda)^{-1}v\},
\end{equation*}
which is an affine step with $M = I$ and $B = [(1\!+\!\lambda)^{-1}\!I \; 0]$.

\paragraph{Fixed-schedule, average-based restarts}
Similar to the momentum steps, fixed-schedule restarts consider a history of $h$ previous iterates, where $h$ is a predefined lookback window~\cite{applegatePracticalLargeScaleLinear2021,applegateFasterFirstorderPrimaldual2023}.
This scheme averages the previous $h$ iterates so that $w = (1/h)\sum_{j=k-h+1}^k g^j$.
So, the restart can be written as
\begin{equation*}
    \phi(v, x) = \{w \in \reals^d \mid w = (1/h) (\ones^T \otimes I) v \},
\end{equation*}
where $v = (g^k, g^{k-1},\dots g^{k-h+1})$.
The operator $\otimes$ represents the Kronecker product and so $\ones^T \otimes I$ is the horizontal stack of $h$ copies of the identity matrix.
This is an affine step with $M=I$, and $B = (1/h) (\ones^T \otimes I)$.

\paragraph{Halpern anchor}
Halpern mechanisms are also known as \emph{anchoring} based methods~\cite{halpernFixedPointsNonexpanding1967, parkExactOptimalAccelerated2022}.
The anchoring step computes a weighted average with the \emph{initial} iterate $s^0$.
So, an anchored gradient step is, given $\beta^k \in (0, 1)$, written as
\begin{equation*}
    w = \beta^k  (g^k - \eta \nabla f(g^k, x)) + (1- \beta^k)s^0.
\end{equation*}
By substituting in the gradient step~\eqref{eq:gd_step} and letting $v=(g^k, s^0)$, we can write
\begin{equation*}
    \phi(v, x) = \{ w \in \reals^d \mid w = \beta^k ((I - \eta P) g^k - \eta x) + (1-\beta^k)s^0\},
\end{equation*}
which is an affine step with $M=I$ and $B = [\beta^k(I\! -\!\eta P)\;\;  (1\! -\!\beta^k)I \; -\!\beta^k\eta I]$.

\subsubsection{Implicit steps}
When $M\neq I$, affine steps represent implicit iterate computations that require linear system solutions with left-hand side matrix $M$.

\paragraph{Proximal operator of a constrained quadratic function}
The proximal operator with $\lambda > 0$ of a convex quadratic function subject to linear equality constraints defined by $A \in \reals^{m\times d}$, $b \in \reals^m$, is given by
\begin{equation}\label{eq:prox_constrained_qp}
    \prox_{\lambda f(\cdot, x)}(v) = \argmin_{\{Az= b\}}~\left\{(1/2) z^T P z + x^T z + 1/(2\lambda) \|z-v\|^2\right\}.
\end{equation}
\errata{We assume that the equalities are consistent and remove redundant rows so that $A$ has full row rank.}
The solution of~\eqref{eq:prox_constrained_qp} can be encoded as a linear system of equations defined by the KKT optimality conditions~\cite[Sec.\ 10.1.1]{boydConvexOptimization2004}.
These include steps in common algorithms such as OSQP~\cite{stellatoOSQPOperatorSplitting2020} and SCS~\cite{ocpb:16,odonoghue:21}, and projections onto affine sets~\cite[Sec.\ 6.2]{parikhProximalAlgorithms2014}.
By introducing dual variable $\nu \in \reals^m$, and defining $y = (v, x, b)$, the KKT conditions corresponding to the \prox~can be written as
\begin{equation*}
    \errata{\phi(v, x) = \left\{(w,\nu) \in \reals^{d + m} \;\middle|\; \begin{bmatrix}
        P + \lambda^{-1} I & A^T \\
        A & 0
    \end{bmatrix}\begin{bmatrix} w \\ \nu\end{bmatrix} = \begin{bmatrix}
        \lambda^{-1} I & -I & 0 \\
        0 & 0 & I
    \end{bmatrix}y
    \right\}},
\end{equation*}
and therefore can be written as an affine step with the corresponding matrices.

\paragraph{Proximal operator of an unconstrained quadratic function}
The proximal operator of a convex quadratic function is a special case of~\eqref{eq:prox_constrained_qp} without the affine constraints and dual variables.
As such, the constraints simplify to~\cite[Sec.\ 6.1.1]{parikhProximalAlgorithms2014}
\begin{equation*}
    \phi(v, x) = \prox_{\lambda f(\cdot, x)} (v) = \{w \in \reals^d \mid (P + \lambda^{-1} I) w = \lambda^{-1}v - x\},
\end{equation*}
which corresponds to an affine step with $M = P + \lambda^{-1} I$ and $B = [\lambda^{-1} I\; -\!I]$.

\paragraph{Projection onto an affine subspace}
The orthogonal projection onto the affine subspace defined by $Az = b$ is also a special case of~\eqref{eq:prox_constrained_qp} with $P = 0$, $x = 0$.
We assume that $A$ is a wide matrix ($m < d$) with linearly independent rows.
By rearranging the terms, this can also be written as an explicit affine step of the form 
\begin{equation*}
    \phi(v, x) = \prox_{\lambda f(\cdot, x)} (v) = \{w \in \reals^d \mid w = (I - A^\dag A)v + A^\dag b\},
\end{equation*}
where $A^\dag = A^T(AA^T)^{-1}$ is the pseudoinverse of $A$.
Thus it also corresponds to an affine step with $M=I$ and $B = [I\!-\!A^\dag\!A\; A^\dag]$.

\subsection{Piecewise affine steps}\label{subsec:piecewise_affine_steps}
We consider elementwise separable, piecewise, affine steps of the form 
\begin{equation}\label{eq:pwa_steps}
    (\phi(v, x))_i = \begin{cases}
        a_1^Ty\errata{{}+b_1}& y \in V_1 \\
        a_2^Ty\errata{{}+b_2}& y \in V_2\\
        \vdots\\
        a_h^Ty\errata{{}+b_h}& y \in V_h
    \end{cases}\quad \text{with} \quad y = (v, x) \in \reals^{d + p},
\end{equation}
for $i=1,\dots,d$, where $a_1,\dots,a_h \in \reals^{d + p}$, \errata{$b_1,\dots,b_h\in\reals$}, and $V_1,\dots, V_h$ are \errata{polyhedral regions covering $\reals^{d + p}$, with formulas that agree on shared boundaries}.
\errata{The scalar activations considered below are nondecreasing in their affine input $a^Ty$, whose coefficients may have either sign.}
To model the piecewise affine behavior over \errata{these} regions of~\eqref{eq:pwa_steps}, we use binary variables.
We derive an exact formulation of the convex hull to the region
\begin{equation}\label{eq:ncvx-step}
 \Phi = \{(y, w) \in [\ubar{y}, \bar{y}] \times \reals \mid w = (\phi(v,x))_i \;\text{with}\;y=(v, x)\},
\end{equation}
where the bounds are $\ubar{y} = (\ubar{v}, \ubar{x}) \le \bar{y} = (\bar{v}, \bar{x})$, for three common steps, the {\it rectified linear unit}, the {\it soft-thresholding}, and the {\it saturated linear unit}.

\paragraph{Rectified linear unit}
We consider the composition of an affine function and the {\it rectified linear unit} operator with the form
\begin{equation}\label{eq:relustep}
    (\phi(v, x))_i = \ReLU(a^Ty) = \max\{a^Ty, 0\},
\end{equation}
where $a \in \errata{\reals^{d+p}}$.
To formulate the step, we must introduce big-$M$ constraints, which are used in tandem with binary variables to make constraints redundant.
Let 
\begin{equation}\label{eq:bigM}
    M^+ = \max_{y \in [\ubar{y}, \bar{y}]}a^Ty\quad \text{and}\quad M^- = \min_{y \in [\ubar{y}, \bar{y}]} a^Ty.
\end{equation}
Assume the constants from~\eqref{eq:bigM} satisfy $M^+ > 0$ and $M^- < 0$, otherwise $\phi$ is just a single linear equality.
We can introduce a binary variable $\omega \in \{0, 1\}$, and write a formulation for~\eqref{eq:relustep} as
\begin{equation*}
    \begin{array}{l}
         w \ge 0, ~ w \ge a^Ty, ~ w \le M^+ \omega,\\
         w \le a^Ty - M^{-}(1-\omega).\\
    \end{array}
\end{equation*}
In this case, the convex hull of~\eqref{eq:ncvx-step} can be directly described using~\cite[Theorem 1]{tjandraatmadjaConvexRelaxationBarrier2020} that we restate here, for completeness, and then extend to the other cases.
\begin{theorem}[Convex hull of rectified linear unit {\cite[Theorem 1]{tjandraatmadjaConvexRelaxationBarrier2020}}]\label{thm:relu_convhull}
Consider the region~$\Phi = \{(y, w) \in [\ubar{y}, \bar{y}] \times \reals \mid w = \ReLU(a^Ty)\}$.
Define vector $\ell^0 \in \errata{\reals^{d+p}}$ with entries $\ell^0_i = \ubar{y}_i$ if $a_i \ge 0$ and $\bar{y}_i$ otherwise.
Similarly, we define vector $u^0 \in \errata{\reals^{d+p}}$ with entries $u^0_i = \bar{y}_i$ if $a_i \ge 0$ and $\ubar{y}_i$ otherwise.
Given $I \subseteq \errata{J = \{i\in\{1,\dots,d+p\}\mid a_i\ne0,\ \ubar{y}_i<\bar{y}_i\}}$, we define the lower bounds of the index set by $\ell_I = \sum_{i \in I} a_i \ell^0_i + \sum_{i \notin I} a_i u^0_i$ and the index set $\mathcal{I} = \{(I, o) \in 2^{J} \times J \mid \ell_I  \ge 0, \; \ell_{I \cup \{o\}} < 0\}$.
\errata{All sums over $i\notin I$ range over $\{1,\dots,d+p\}\setminus I$, so fixed coordinates retain their constant contributions.}
Note that $\ell_{\emptyset}$ ($\emptyset$ denotes the empty set) corresponds to the upper bound of $a^Ty$ over $[\ubar{y}, \bar{y}]$.
The convex hull of $\Phi$ is
\begin{equation}
    \conv(\Phi) = \left\{(y, w) \in [\ubar{y}, \bar{y}] \times \reals \;\middle|\; \begin{cases}
        w = a^Ty & \ell_J \ge 0\\
        w = 0 & \ell_\emptyset < 0\\
        (y, w) \in Q & \text{otherwise}
    \end{cases}\right\},
\end{equation}
with
\begin{equation*}
    Q =\! \left\{(y, w) \in \reals^{d + p + 1} \,\middle|
    \!\!\!\!\begin{array}{ll}
    &w \ge a^Ty,\quad w \ge 0\\ 
    &w \le \sum_{i \in I} a_i(y_i - \ell^0_i) + \frac{\ell_I}{u^0_o - \ell^0_o} (y_o - \ell^0_o),\; \forall (I, o) \in \mathcal{I}\\
    \end{array}\right\}.
\end{equation*}
\end{theorem}
Note that while the number of inequalities for the convex hull in Theorem~\ref{thm:relu_convhull} is exponential in $d$, there exists an efficient separation oracle that either verifies if a candidate point $(y, w) \in \conv(\Phi)$ or gives a violated inequality~\cite[Proposition 2]{tjandraatmadjaConvexRelaxationBarrier2020}.
These steps appear when composing affine steps and projections onto the nonnegative orthant. For example, {\it projected gradient descent} with step size $\eta > 0$ applied to problem~\eqref{eq:prob} with $q(x) = x$, $A=-I$, $b(x) = 0$, and $C = \reals_+^n$, becomes $s^{k+1} = \Pi ((I - \eta P)s^k - \eta x)$ where $\Pi = \ReLU$ clips the negative entries of its argument to $0$. Therefore, with $v=s^k$, we can write the projected gradient descent iterations as 
\begin{equation*}
    (\phi(v, x))_i = \ReLU(((I - \eta P)v - \eta x)_i),
\end{equation*}
which corresponds to equation~\eqref{eq:relustep} with $a$ being the $i$th row of $[I\!-\!\eta P\;\; -\eta I]$.

\begin{figure}[t!]
\centering
\begin{tikzpicture}[scale=0.12]
    \fill[purple!20, semitransparent] (8,17) -- (17, 8)  -- (24, 9) -- (15,18) -- cycle;
    \fill[cyan!35] (11, 14) -- (2,2) -- (14, 11) -- cycle;
    \fill[cyan!20] (11, 14) -- (14, 15) -- (18, 15) -- (21, 12) -- (18, 11) -- (14, 11) -- cycle;
    \fill[cyan!25] (18, 15) -- (30, 24) -- (21, 12) -- cycle;
    \draw[->, thin, opacity=0.3] (4, 13) -- (30, 13) node[below] {$y_1$};
    \draw[->, thin, opacity=0.3] (1, 8) -- (31, 18) node[above] {$y_2$};
    \draw[->, thin, opacity=0.3] (16, 4) -> (16, 25) node[left] {$w$};
    \draw[thick, dotted, purple] (2,11) -- (18, 11) -- (30,15) -- (14,15) -- (2,11);
    \fill[purple!10, nearly transparent] (2,11) -- (18, 11) -- (30,15) -- (14,15) -- (2,11);
    \draw[thick, purple] (8,17) -- (17, 8);
    \draw[thick, purple] (15,18) -- (24, 9);
    \draw[thick, myblue] (18,11) -- (14, 11);
    \draw[thick, myblue] (18,11) -- (21, 12);
    \draw[thick, myblue] (11,14) -- (14, 15);
    \draw[thick, myblue] (14,15) -- (18,15);
    \draw[thick, myblue] (2,2) -- (14, 11);
    \draw[thick, myblue] (2,2) -- (11,14);
    \draw[thick, myblue] (30,24) -- (21, 12);
    \draw[thick, myblue] (30,24) -- (18,15);
    \node[purple] at (5, 19) {$y_1 + y_2 + \lambda = 0$};
    \node[purple] at (26, 7) {$y_1 + y_2 - \lambda = 0$};
    \node[myblue] at (30, 26) {$(\bar{y}_1, \bar{y}_2, \bar{y}_1 + \bar{y}_2 - \lambda)$};
    \fill[myblue] (30, 24) circle (7pt);
    \node[myblue] at (3, 1) {$(\ubar{y}_1, \ubar{y}_2, \ubar{y}_1 + \ubar{y}_2 + \lambda)$};
    \fill[myblue] (2, 2) circle (7pt);

    \fill[purple] (11, 14) circle (7pt);
    \fill[purple] (14, 11) circle (7pt);
    \fill[myblue] (18, 11) circle (7pt);
    \fill[myblue] (14, 15) circle (7pt);
    \fill[purple] (21, 12) circle (7pt);
    \fill[purple] (18, 15) circle (7pt);
\end{tikzpicture}
\begin{tikzpicture}[scale=0.12]
    \fill[cyan, transparent] (18, 15) -- (30, 24) -- (21, 12) -- cycle;
    \fill[cyan, nearly transparent] (11, 14) -- (2,2) -- (14, 11) -- cycle;
    
    \fill[cyan!85, nearly transparent] (2,2) -- (18,11) -- (14, 11) -- cycle;
    \fill[cyan, nearly transparent] (2,2) -- (18,11) -- (21, 12) -- cycle;
    \fill[cyan, nearly transparent] (30,24) -- (18,11) -- (21, 12) -- cycle;
    \fill[cyan!80, nearly transparent] (30,24) -- (18,11) -- (14,11) -- cycle;
    \fill[cyan!80, nearly transparent] (30,24) -- (14,11) -- (11,14) -- cycle;
\draw[->, thin, opacity=0.3] (4, 13) -- (30, 13) node[below] {$y_1$};
    \draw[->, thin, opacity=0.3] (1, 8) -- (31, 18) node[above] {$y_2$};
    \draw[->, thin, opacity=0.3] (16, 4) -> (16, 25) node[left] {$w$};
    \draw[thick, purple] (11,14) -- (14, 11);
    \draw[thick, dashed, purple] (18,15) -- (21, 12);
    \draw[thick, myblue] (18,11) -- (14, 11);
    \draw[thick, myblue] (18,11) -- (21, 12);
    \draw[thick, dashed, myblue] (11,14) -- (14, 15);
    \draw[thick, dashed, myblue] (14,15) -- (18,15);
    \draw[thick, myblue] (2,2) -- (14, 11);
    \draw[myblue] (2,2) -- (21, 12);
    \draw[myblue] (2,2) -- (18,11);
    \draw[thick, myblue] (2,2) -- (11,14);
    \draw[dashed, myblue] (2,2) -- (14,15);
    \draw[dashed, myblue] (2,2) -- (18,15);
    \draw[thick, myblue] (30,24) -- (21, 12);
    \draw[myblue] (30,24) -- (18,11);
    \draw[myblue] (30,24) -- (14, 11);
    \draw[myblue] (30,24) -- (11, 14);
    \draw[thick, myblue] (30,24) -- (21, 12);
    \draw[thick, dashed, myblue] (30,24) -- (18, 15);
    \draw[dashed, myblue] (30,24) -- (14, 15);
    \fill[myblue] (30, 24) circle (7pt);
    \fill[purple] (11, 14) circle (7pt);
    \fill[purple] (14, 11) circle (7pt);
    \fill[myblue] (18, 11) circle (7pt);
    \fill[myblue] (14, 15) circle (7pt);
    \fill[purple] (21, 12) circle (7pt);
    \fill[purple] (18, 15) circle (7pt);
    \node[myblue] at (28, 26) {$\hphantom{(\bar{y}_1, \bar{y}_2, \bar{y}_1 + \bar{y}_2 - \lambda)}$};
    \node[myblue] at (3, 1) {$\hphantom{(\ubar{y}_1, \ubar{y}_2, \ubar{y}_1 + \ubar{y}_2 + \lambda)}$};
\end{tikzpicture}
\caption{Left: Example of soft-thresholding $w = \mathcal{T}_\lambda(y_1 + y_2)$ with $\lambda = 0.5$ variable bounds $\ubar{y}_1 \leq y_1 \leq \bar{y}_1$ and $\ubar{y}_2 \leq y_2 \leq \bar{y}_2$; the two lines $y_1 + y_2 + \lambda = 0$ and $y_1 + y_2 - \lambda = 0$ delimit the area where $w = 0$.
Right: the convex hull of $w= \phi_\lambda(y_1 + y_2)$ given by Theorem~\ref{thm:softthreshold_convhull}.
\label{fig:soft3d}}
\end{figure}

\paragraph{Soft-thresholding}
We consider the composition of an affine function and the {\it soft-thresholding} operator with the form
    \begin{equation}\label{eq:soft-thresh}
   \!\!\!\!(\phi(v, x))_i = \mathcal{T}_\lambda(a^Ty) =
   \begin{cases}
        a^T y - \lambda & {a}^T {y} > \lambda \\
        0 & |{a}^T {y}| \le \lambda \\
        {a}^T {y} + \lambda & {a}^T {y} < -\lambda
   \end{cases}\quad \text{with}\quad y = (v, x),
\end{equation}
where $\mathcal{T}_{\lambda}$ is the {\it soft-thresholding} operator with constant $\lambda > 0$.

By introducing binary variables $\omega, \zeta \in \{0, 1\}$ and defining $M^+, M^-$ as in~\eqref{eq:bigM}, we can formulate~\eqref{eq:soft-thresh} as
\begin{align*}
         a^Ty - \lambda \le w &\le a^T y + \lambda \\
         a^Ty + \lambda - 2\lambda(1 - \zeta) \le w &\le a^Ty - \lambda + 2\lambda(1-\omega)\\
         (M^- + \lambda)\zeta \le w &\le (M^+ - \lambda)\omega\\
         \lambda + (M^- - \lambda)(1-\omega) \le a^Ty &\le \lambda + (M^+ - \lambda)\omega\\
         -\lambda + (M^- + \lambda)\zeta \le a^Ty &\leq -\lambda + (M^+ + \lambda)(1-\zeta). 
\end{align*}
Note that, if $M^+ < \lambda$ or $M^- > -\lambda$, then we can prune at least one of the binary variables by setting it to zero.

In order to construct the convex hull of~\eqref{eq:soft-thresh}, we extend~\cite[Theorem 1]{tjandraatmadjaConvexRelaxationBarrier2020} from
the convex hull of the piecewise \emph{maximum} of affine functions to the convex hull of nonconvex monotonically nondecreasing functions such as the soft-thresholding function. 
Figure \ref{fig:soft3d} illustrates the operator $w = \mathcal{T}_\lambda(y_1 + y_2)$ with $\lambda = 0.5$ with its convex hull.
We can now obtain the convex hull of~\eqref{eq:soft-thresh} with the following Theorem.
\begin{theorem}[Convex hull of soft-thresholding operator]\label{thm:softthreshold_convhull}
    Define $\ell^0, u^0, \ell_I, J$ as done in Theorem~\ref{thm:relu_convhull}.
    Additionally, given an index set $I$, define the upper bounds $u_I = \sum_{i \in I} a_i u^0_i + \sum_{i \notin I} a_i \ell^0_i$.
    Note that $u_{\emptyset}$ corresponds to the lower bound of $a^Ty$ over $[\ubar{y}, \bar{y}]$ and by definition $u_{\emptyset} = \ell_J$ and $u_J = \ell_\emptyset$.
    Also, define the sets of indices $\mathcal{I}^{+} = \{(I, o) \in 2^{J} \times J \mid \ell_I  \ge \lambda, \; \ell_{I \cup \{o\}} < \lambda\}$ and $\mathcal{I}^{-} = \{(I, o) \in 2^{J} \times J \mid u_I  \le -\lambda, \; u_{I \cup \{o\}} > -\lambda\}$.
    \errata{Finally, define the families of index sets $\mathcal{J}^{+} = \{I \subseteq J \mid \ell_I \ge \lambda\}$ and $\mathcal{J}^{-} = \{I \subseteq J \mid u_I \le -\lambda\}$.}
Consider the region~$\Phi = \{(y, w) \in [\ubar{y}, \bar{y}] \times \reals \mid w = \mathcal{T}_\lambda(a^Ty)\}$. We can write its convex hull as
\begin{equation}
    \conv(\Phi) = \left\{(y, w) \in [\ubar{y}, \bar{y}] \times \reals \;\middle|\; \begin{cases}
        w = a^Ty - \lambda & \ell_J > \lambda \\
        w = a^Ty + \lambda & u_J < -\lambda \\
        w = 0 & -\lambda \le u_\emptyset \le \ell_\emptyset \le \lambda \\
        (y, w) \in Q & \text{otherwise}
    \end{cases}\right\},
\end{equation}
where 
\begin{equation*}
    Q =\! \left\{(y, w) \in \reals^{d + p + 1} \,\middle|
    \!\!\!\!\begin{array}{ll}
    &a^Ty - \lambda \le w \le a^Ty + \lambda\\ 
    &\errata{w \le 0 \;\text{ if } u_J \le \lambda, \qquad w \ge 0 \;\text{ if } \ell_J \ge -\lambda}\\
    &\errata{w \le \frac{\ell_I - \lambda}{\ell_I + \lambda} (a^Ty + \lambda) + \frac{2\lambda}{\ell_I + \lambda} \sum_{i \in I} a_i(y_i - \ell^0_i),\; \forall I \in \mathcal{J}^+}\\
    &\errata{w \ge \frac{u_I + \lambda}{u_I - \lambda} (a^Ty - \lambda) + \frac{2\lambda}{\lambda - u_I} \sum_{i \in I} a_i(y_i - u^0_i),\; \forall I \in \mathcal{J}^-}\\
    &w \le \sum_{i \in I} a_i(y_i - \ell^0_i) + \frac{\ell_I - \lambda}{u^0_o - \ell^0_o} (y_o - \ell^0_o),\; \forall (I, o) \in \mathcal{I}^+\\
    &w \ge \sum_{i \in I} a_i(y_i - u^0_i) + \frac{u_I + \lambda}{\ell^0_o - u^0_o} (y_o - u^0_o),\; \forall (I, o) \in \mathcal{I}^-\end{array}\right\}.
\end{equation*}
\end{theorem}

\errata{
\begin{proof}
    We consider the last case in $\conv(\Phi)$, since $\mathcal{T}_\lambda (a^T y)$ is affine over the box in the first three.
    Let $\theta = a^Ty$ and $\delta_I(y) = \sum_{i \in I} a_i(y_i - \ell^0_i) \ge 0$.
    For $I \in \mathcal{J}^{+}$, let $\rho_I = (\ell_I - \lambda)/(\ell_I + \lambda) \in [0, 1)$, so that the right-hand side of line~3 of $Q$ is $R_I(y) = \rho_I(\theta + \lambda) + (1 - \rho_I)\delta_I(y)$.
    The set $\conv(\Phi)$ consists of the points $(y, w)$ with $y \in [\ubar{y}, \bar{y}]$ and $w$ between the convex and the concave envelopes of $\mathcal{T}_\lambda(a^Ty)$ over the box~\cite[Appendix~A.1]{tjandraatmadjaConvexRelaxationBarrier2020}.
    We show that the concave envelope is the minimum of the upper bounds in lines~1, 2, 3, and~5 of $Q$, and the lower bounds follow by applying the same argument to $-w = \mathcal{T}_\lambda(-a^Ty)$ (lines~1, 2, 4, and~6 of $Q$).
    
    To show validity of the formulation, first observe that lines~1, 2, and~5 of $Q$ hold because $\mathcal{T}_\lambda(\theta) \le \theta + \lambda$ and $\mathcal{T}_\lambda(\theta) \le \ReLU(\theta - \lambda)$, using Theorem~\ref{thm:relu_convhull} applied to $\ReLU(a^Ty - \lambda)$.
    For line~3, write $\mathcal{T}_\lambda(\theta) = \max\{\theta - \lambda, \min\{\theta + \lambda, 0\}\}$.
    We have $R_I(y) \ge \rho_I(\theta + \lambda) \ge \min\{\theta + \lambda, 0\}$ and, because $\lambda(1 + \rho_I) = (1 - \rho_I)\ell_I$, also $R_I(y) - (\theta - \lambda) = (1 - \rho_I)(\ell_I - \theta + \delta_I(y)) \ge 0$, where the last inequality holds because $\theta - \delta_I(y) \le \ell_I$ over the box.

    For sufficiency of the formulation, recall that the concave envelope is the infimum of the affine upper bounds $L$ on $\mathcal{T}_\lambda (a^T y)$ over the box.
    An affine function $L$ satisfies $L \ge \min\{\theta + \lambda, 0\}$ over the box if and only if there exists a $\rho \in [0,1]$ such that $L \ge \rho(\theta + \lambda)$ over the box.
    To justify this step, write $u(y)=\theta+\lambda$ and choose $\rho=\sup(\{0\}\cup\{L(y)/u(y)\mid y\in[\ubar{y},\bar{y}],\ u(y)<0\})$.
    The assumed bound gives $0\le\rho\le1$.
    For points $y,z$ with $u(y)<0<u(z)$, their convex combination with $u=0$ has $L\ge0$, which implies $L(y)/u(y)\le L(z)/u(z)$.
    Thus $\rho u\le L$ also where $u>0$, and the cases $u\le0$ follow directly.
    The converse follows from $\min\{u,0\}\le\rho u$.
    Therefore, the concave envelope of $\mathcal{T}_\lambda(a^T y)$ is the infimum over $\rho \in [0, 1]$ of the concave envelope of $f_\rho(y) = \max\{\theta - \lambda, \rho(\theta + \lambda)\}$.
    We have $f_1 = \theta + \lambda$ (line~1 of $Q$) and, for $\rho < 1$, $f_\rho = \rho(\theta + \lambda) + (1 - \rho)\ReLU(\theta - \sigma_\rho)$ with $\sigma_\rho = \lambda (1 + \rho)/(1 - \rho) \ge \lambda$.
    Applying Theorem~\ref{thm:relu_convhull} to $\ReLU(a^Ty - \sigma_\rho)$ shows that the concave envelope of $f_\rho$ is $\rho(\theta + \lambda)$ if $u_J < \sigma_\rho$, it is $f_\rho = \theta - \lambda = R_J(y)$ if $\ell_J \ge \sigma_\rho$, and otherwise it is the minimum of $\rho(\theta+\lambda)+(1-\rho)\delta_I(y) + (1-\rho)(\ell_I - \sigma_\rho)(y_o - \ell_o^0)/(u_o^0 - \ell_o^0)$ over the pairs $(I, o)$ with $\ell_I \ge \sigma_\rho > \ell_{I \cup \{o\}}$.
    For fixed $y$, each expression is affine in $\rho$ because $(1-\rho)\sigma_\rho = \lambda(1+\rho)$, and applies for $\rho$ in an interval, so the infimum is its value at an endpoint of the closure of the interval.
    The endpoint $\rho = 1$ gives line~1, $\rho = 0$ gives line~2 or~5, $\sigma_\rho = \ell_I$ gives $R_I(y)$, and $\sigma_\rho = \ell_{I \cup \{o\}}$ gives $R_{I \cup \{o\}}(y)$ because $\ell_I - \ell_{I \cup \{o\}} = a_o (u_o^0 - \ell_o^0)$.
    Therefore, the concave envelope is at least the minimum of the upper bounds in lines~1, 2, 3, and~5 of $Q$, and is equal to it by our validity argument.
\end{proof}
}
Even though the constraints in~$Q$ define exponentially many inequalities, we can define an efficiently solvable separation problem: given a point $(\hat{y}, \hat{w})$, verify that $(\hat{y}, \hat{w}) \in \errata{\conv(\Phi)}$, otherwise return a violated inequality.
\begin{proposition}\label{prop:separation_procedure}
The separation problem for $\conv(\Phi)$ is solvable in $O(d + p)$ ~\!time. 
\end{proposition}
\errata{
\begin{proof}
First check the box constraints, the applicable affine cases, and lines~1 and~2 of $Q$.
Separate the families $\mathcal{I}^{+}$ and $\mathcal{I}^{-}$ in $O(d+p)$ time by the fractional-knapsack procedure of~\cite[Prop.~2]{tjandraatmadjaConvexRelaxationBarrier2020}.
It remains to check $\mathcal{J}^{+}$ and $\mathcal{J}^{-}$.
For the upper family, write $d_i=a_i(u_i^0-\ell_i^0)>0$, $x_i=(\hat y_i-\ell_i^0)/(u_i^0-\ell_i^0)\in[0,1]$, $U=u_J$, and $\theta=a^T\hat y$.
If $U<\lambda$, this family is empty.
Otherwise set $C=U-\lambda$ and $A=(U-\lambda)(\theta+\lambda)-\hat w(U+\lambda)$.
For $I\in\mathcal J^+$, multiplying the slack of its inequality by $\ell_I+\lambda>0$ gives
\[
 A+\sum_{i\in I}d_i\bigl(2\lambda x_i-(\theta+\lambda-\hat w)\bigr),
 \qquad \sum_{i\in I}d_i\le C.
\]
The set $I_* = \{i\in J\mid 2\lambda x_i<\theta+\lambda-\hat w\}$ minimizes this expression without the constraint on the sum.
If $I_*$ is admissible, its inequality therefore detects any violation in $\mathcal J^+$.
If $I_*$ is inadmissible, interpolate any feasible $z\in[0,1]^J$ with the indicator of $I_*$ until $\sum_i d_i z_i=C$.
The cleared slack cannot increase, because the indicator of $I_*$ minimizes it over $[0,1]^J$.
At equality its objective is
\[
 2\lambda\left(\sum_i d_i x_i z_i-\hat w\right).
\]
The already checked ReLU upper bounds give $\hat w\le\min\{\sum_i d_i x_i z_i\mid 0\le z\le1,\ \sum_i d_i z_i=C\}$ by the same fractional-knapsack characterization.
Consequently, no inequality in $\mathcal J^+$ is violated in this case.
Reflection $(a,\hat w)\mapsto(-a,-\hat w)$ handles $\mathcal J^-$.
Each additional check takes one pass through the coordinates, so the total time is $O(d+p)$.
If a check fails, it supplies a violated inequality.
\end{proof}
}
\ifsiam\else\begin{remark}\fi
We can also solve the separation problem in $O((d+p)\log(d+p))$ time by sorting the elements \errata{indexed by $J$} in nondecreasing order of $(\hat{y}_i - \ell^0_i)/(u^0_i - \ell^0_i)$ (resp.\ $(\hat{y}_i - u^0_i)/(\ell^0_i - u^0_i)$) and by iteratively including them to the set $I$ until $\ell_I - \lambda < 0$ (resp. $u_I + \lambda > 0$).
The last set $I$ and index $o$ that triggers the change of sign gives the most violated inequality \errata{of $Q$ over $\mathcal{I}^{+}$ (resp.\ $\mathcal{I}^{-}$)}~\cite[Prop. 2]{tjandraatmadjaConvexRelaxationBarrier2020}.
\errata{The additional families are then checked using $I_*$ and its reflected counterpart as in the proof above.}
\ifsiam\else\end{remark}\fi
These steps appear when combining affine steps and proximal operators of nonsmooth functions. For example, the proximal gradient method with step size $\eta> 0$ applied to a Lasso~\cite{lasso} function of the form $f(z, x) = \errata{(1/2)}\|Dz - x\|_2^2 + \lambda \|z\|_1$ with $D\in \reals^{p \times n}$ becomes the \emph{iterative shrinkage thresholding algorithm} (ISTA)~\cite[Ch.\ 10]{beckFirstOrderMethodsOptimization2017} with iterations defined as $s^{k+1} = \mathcal{T}_{\lambda \eta} ( (I - \eta D^TD) s^{k} + \eta D^Tx)$, where $\mathcal{T}_{\lambda \eta}$ applies the soft-thresholding operator elementwise. Therefore, with $v = s^k$, we can write ISTA iterations as
\begin{equation}\label{eq:istastep}
    (\phi(v, x))_i = \mathcal{T}_{\lambda \eta} ( ( (I - \eta D^TD) v + \eta D^Tx)_i),
\end{equation}
which corresponds to equation~\eqref{eq:soft-thresh} with $a$ being the $i$th row of $[I\!-\!\eta D^T\!D\;\; \eta D^T]$.

\paragraph{Saturated linear unit}
Given two scalars $b, c \in \reals$ \errata{with $b\le c$}, we consider the composition of an affine function and the \emph{saturated linear unit} with the form
\begin{equation}
    \label{eq:sat-lin}
   \!\!\!\!(\phi(v, x))_i = \mathcal{S}_{[b, c]}(a^Ty) =
   \begin{cases}
        b & {a}^T {y} < b \\
        a^T y & b \le a^T y \le c \\
        c  & {a}^T {y} > c
   \end{cases}\quad \text{with}\quad y = (v, x).
\end{equation}
\errata{If $b=c$, the operator is constant and its convex hull is simply $w=b$ over the box.}
For the big-$M$ formulation, we again introduce binary variables $\omega, \zeta \in \{0, 1\}$ arising from the three cases, and assume that the constants from~\eqref{eq:bigM} are $M^- < b$ and $M^+ > c$.
\errata{Otherwise, replace wholly saturated cases by $w=b$ or $w=c$, and remove branches whose input regions do not intersect $[M^-,M^+]$.}
We can write~\eqref{eq:sat-lin} as
\begin{align*}
         b \le w &\le c \\
         c - (c-b)(1- \omega) \le w &\le b + (c - b)(1-\zeta)\\
         a^Ty - (M^+ - M^-)\omega\le w &\le a^Ty + (M^+-M^-)\zeta\\
         c + (M^- - c)(1-\omega) \le a^Ty &\le c + (M^+ - c)\omega \\
         b + (M^- - b)\zeta \leq a^Ty &\leq b + (M^+ - b)(1-\zeta).
\end{align*}
Similarly to the soft-thresholding case, the saturated linear unit is a nonconvex monotonically nondecreasing function, so we again extend the results of~\cite[Theorem 1]{tjandraatmadjaConvexRelaxationBarrier2020} and construct the convex hull of~\eqref{eq:sat-lin}.
Although we use a technique similar to that in the soft-thresholding case, we need to slightly redefine the variables from before, as shown in Theorem~\ref{thm:satlin_convhull}.
\begin{theorem}[Convex hull of saturated linear unit]\label{thm:satlin_convhull}
For scalars $b, c \in \reals$ \errata{with $b < c$}, and variable $y \in \errata{\reals^{d+p}}$, consider the region
\begin{equation*}
    \Phi = \left\{(y, w) \in [\ubar{y}, \bar{y}] \times [b, c] \mid w = \mathcal{S}_{[b, c]}(a^T y)\right\}.
\end{equation*}
Let $\ell_I, u_I, \ell_i^0, u_i^0, J$ be defined in the same manner as for the soft-thresholding operator in Theorem~\ref{thm:softthreshold_convhull}, but redefine the index sets $\mathcal{I}^{+} = \{(I, o) \in 2^{J} \times J \mid \ell_I  \ge b, \; \ell_{I \cup \{o\}} < b\}$ and $\mathcal{I}^{-} = \{(I, o) \in 2^{J} \times J \mid u_I  \le c, \; u_{I \cup \{o\}} > c\}$.
\errata{Also, redefine the families of index sets $\mathcal{J}^{+} = \{I \subseteq J \mid \ell_I \le b\}$ and $\mathcal{J}^{-} = \{I \subseteq J \mid u_I \ge c\}$.}
We can express $\conv(\Phi)$ as
\begin{equation*}
    \conv(\Phi) = \left\{ (y, w) \in [\ubar{y}, \bar{y}] \times [b, c] \;\middle|\; \begin{cases}
        w = b & u_J < b\\
        w = c & \ell_J > c\\
        w = a^T y & b \le u_\emptyset \le \ell_\emptyset \le c \\
        (y, w) \in Q & \text{otherwise}
    \end{cases}
    \right\},
\end{equation*}
where
\begin{equation*}
    Q = \left\{ (y, w) \in \reals^{p+d+1} \! \;\middle|\; \!\!\!\!\!\!\!\! \begin{array}{ll}
    &b\le w \le c \\
    &\errata{w \le a^Ty \;\text{ if } \ell_J \ge b, \qquad w \ge a^Ty \;\text{ if } u_J \le c}\\
    &\errata{w \le b + \frac{c-b}{c-\ell_I} \sum_{i \in I} a_i(y_i - \ell^0_i),\; \forall I \in \mathcal{J}^+}\\
    &\errata{w \ge c + \frac{c-b}{u_I-b} \sum_{i \in I} a_i(y_i - u^0_i),\; \forall I \in \mathcal{J}^-}\\
    &w \le \sum_{i \in I} a_i(y_i - \ell^0_i) + \frac{\ell_I - b}{u^0_o - \ell^0_o} (y_o - \ell^0_o) + b,\; \forall (I, o) \in \mathcal{I}^+\\
    &w \ge \sum_{i \in I} a_i(y_i - u^0_i) + \frac{u_I-c}{\ell^0_o - u^0_o} (y_o - u^0_o) + c,\; \forall (I, o) \in \mathcal{I}^-
    \end{array}\!\!\!
    \right\}.
\end{equation*}
\end{theorem}
\errata{
\begin{proof}
    The proof is analogous to the proof of Theorem~\ref{thm:softthreshold_convhull}, so we only state the differences for~\eqref{eq:sat-lin}.
    Let $\theta = a^Ty$ and $\delta_I(y) = \sum_{i \in I} a_i(y_i - \ell^0_i) \ge 0$.
    For $I \in \mathcal{J}^+$, let $\kappa_I = (c - b) / (c - \ell_I) \in (0, 1]$, so that the right-hand side of line~3 of $Q$ is $R_I(y) = b + \kappa_I \delta_I(y)$.
    The identity $\mathcal{S}_{[b, c]}(\theta) = \max\{b, \min\{\theta, c\}\}$ replaces the one for $\mathcal{T}_\lambda$, and the lower bounds follow by applying the argument to $-w = \mathcal{S}_{[-c, -b]}(-\theta)$.

    For validity, lines~1, 2, and~5 of $Q$ hold because $\mathcal{S}_{[b, c]}(\theta) \le c$, $\mathcal{S}_{[b, c]}(\theta) \le \theta$ when $\ell_J \ge b$, and $\mathcal{S}_{[b, c]}(\theta) \le b + \ReLU(\theta - b)$, using Theorem~\ref{thm:relu_convhull} applied to $\ReLU(a^Ty - b)$.
    For line~3, we have $R_I(y) \ge b$ and, because $\kappa_I (c - \ell_I) = c - b$ and $\theta - \delta_I(y) \le \ell_I$ over the box, we also have $R_I(y) - (\kappa_I \theta + (1-\kappa_I)c) = \kappa_I(c - \theta + \delta_I(y)) - (c - b) \ge 0$, where $\kappa_I \theta + (1 - \kappa_I)c \ge \min\{\theta, c\}$.

    Lastly, for sufficiency, apply the affine-majorant argument in the preceding proof to $L-c$ and $\theta-c$.
    It gives that $L \ge \min\{\theta,c\}$ over the box if and only if there exists $\rho\in[0,1]$ such that $L\ge\rho\theta+(1-\rho)c$ over the box.
    So, the concave envelope of $\mathcal{S}_{[b, c]}(a^T y)$ is the infimum over $\rho \in [0, 1]$ of the concave envelope of $g_\rho(y) = \max\{b, \rho \theta + (1- \rho)c \}$.
    We have $g_0 = c$ (line~1 of $Q$), and, for $\rho \in (0, 1]$, $g_\rho = b + \rho \ReLU(\theta - \tau_\rho)$, with $\tau_\rho = c - (c - b)/\rho \le b$.
    Applying Theorem~\ref{thm:relu_convhull} to $\ReLU(a^Ty - \tau_\rho)$ shows that the concave envelope of $g_\rho$ is $b = R_\emptyset(y)$ if $u_J < \tau_\rho$, it is $\rho\theta + (1-\rho)c$ if $\ell_J \ge \tau_\rho$, and otherwise it is the minimum of $b + \rho \delta_I(y) + \rho(\ell_I - \tau_\rho)(y_o - \ell_o^0)/(u_o^0 - \ell_o^0)$ over the pairs $(I, o)$ with $\ell_I \ge \tau_\rho > \ell_{I \cup \{o\}}$.
    These expressions are affine in $\rho$ because $\rho\tau_\rho = b - (1-\rho)c$, so their infima are attained at the endpoints.
    The endpoint $\rho \to 0$ gives line~1, $\rho = 1$ gives line~2 or line~5, $\tau_\rho = \ell_I$ gives $R_I(y)$, and $\tau_\rho = \ell_{I \cup \{o\}}$ gives $R_{I \cup \{o\}}(y)$.
\end{proof}
}
\ifsiam\else\begin{remark}\fi
    \errata{Separation again takes $O(d+p)$ time.
    Check the box, affine cases, lines~1 and~2, and the ReLU pair families first.
    If $\ell_J>b$, the upper subset family is empty and can be skipped.
    With $d_i,x_i,U$ as in Proposition~\ref{prop:separation_procedure}, the cleared slack for $I\in\mathcal J^+$ is
    \[
      -(\hat w-b)(c-U)+\sum_{i\in I}d_i\bigl((c-b)x_i-(\hat w-b)\bigr),
      \qquad \sum_{i\in I}d_i\ge U-b.
    \]
    Thus use $I_* = \{i\in J\mid(c-b)x_i<\hat w-b\}$.
    If it is admissible, it minimizes the cleared slack over all admissible subsets.
    Otherwise, interpolate any feasible $z\in[0,1]^J$ with the indicator of $I_*$ until $\sum_i d_i z_i=U-b$.
    This cannot increase the cleared slack, whose boundary value is $(c-b)(\sum_i d_i x_i z_i-(\hat w-b))\ge0$ by the already checked ReLU bounds.
    Reflection $(a,\hat w,b,c)\mapsto(-a,-\hat w,-c,-b)$ handles the lower family.}
\ifsiam\else\end{remark}\fi

These steps appear when composing affine steps and projections onto boxes.
For example, consider a QP of the form~\eqref{eq:prob} with $q = x$, $A=-I$, $b(x) = 0$, and $C = [b, c]$ for $b, c \in \reals^n$.
Similarly to the ReLU case, a box-constrained QP of this form can be solved again with projected gradient descent with step size $\eta > 0$ with iterations $s^{k+1} = \mathcal{S}_{[b, c]}((I-\eta P)s^k - \eta x)$, where $\mathcal{S}_{[b, c]}$ applies~\eqref{eq:sat-lin} elementwise.
Therefore, with $v = s^k$, we can write projected gradient descent iterations as
\begin{equation}
    (\phi(v, x))_i = \mathcal{S}_{[b_i, c_i]}\left(((I-\eta P)v - \eta x)_i\right),
\end{equation}
which corresponds to equation~\eqref{eq:sat-lin} with $a$ being the $i$th row of $[I\!-\!\eta P\;\; -\eta I]$.

\subsection{Vanilla MILP formulation for ISTA}\label{ex:vanillamilp}
Consider the problem of recovering a sparse vector $z^\star$ from its noisy linear measurements $x = Dz^\star + \varepsilon \in \reals^p$, where $D \in \reals^{p \times n}$ is the dictionary matrix and $\varepsilon \in \reals^p$.
A popular approach is to model this problem using Lasso~\cite{lasso} formulation with sparsity-inducing regularization parameter $\lambda > 0$, \ie,
\begin{equation}\label{prob:lasso}
	\begin{array}{ll}
		\text{minimize} & (1/2)\|Dz - x\|_2^2 + \lambda \| z \|_1.
	\end{array}
\end{equation}
The sparse coding problem~\eqref{prob:lasso} can be written as
\begin{equation}\label{eq:lasso_qpprob}
    \begin{array}{ll}
        \text{minimize} & (1/2)z^TD^TDz -(D^Tx)^Tz + \lambda \ones^T \mu\\
        \text{subject to} & -\mu \le z \le \mu,
    \end{array}
\end{equation}
which corresponds to problem~\eqref{eq:prob} with 
\begin{equation*}
    P = \begin{bmatrix} D^TD  & \\ & 0\end{bmatrix} \in \reals^{2n \times 2n},\quad A = \begin{bmatrix}I &-I\\ -I & -I\end{bmatrix} \in \reals^{2n \times 2n}, 
\end{equation*}
$q(x) = (-D^\tpose x, \lambda \ones) \in \reals^{2n}$, $b(x) = 0 \in \reals^{2n}$, and $C = \reals_+^{2n}$.
We show (Appendix~\ref{apx:Hmatrix_proofs}) that $\errata{H = D^TD - \eta^{-1}I}$ is the scaling matrix corresponding to the optimality conditions.

We choose $p=15$, $n=20, \lambda=10^{-1}$, and, similarly to~\cite[Sec.\ 4.1]{chenTheoreticalLinearConvergence2018}, we sample the entries of $D$ i.i.d.\ from the Gaussian distribution $\mathcal{N}(0, 1/m)$ with $20\%$ nonzeros and normalize its columns.
We generate 100 sparse vectors $z^\star$ with $z^\star_i \sim \mathcal{N}(0, 1)$ with probability 0.1 and otherwise $z^\star_i = 0$. We chose $X$ to be the tightest hypercube containing all samples.
We set  $S = \{s^0\}$ where $s^0$ is the solution of the unregularized problem ($\lambda = 0$) for a random sample. 
We solve this problem with ISTA iterations of the form~\eqref{eq:istastep} with step size $\eta = 1/L$ where $L$ is the largest eigenvalue of $D^TD$.
\begin{figure}[t!]
    \centering
    \includegraphics[width=0.85\textwidth]{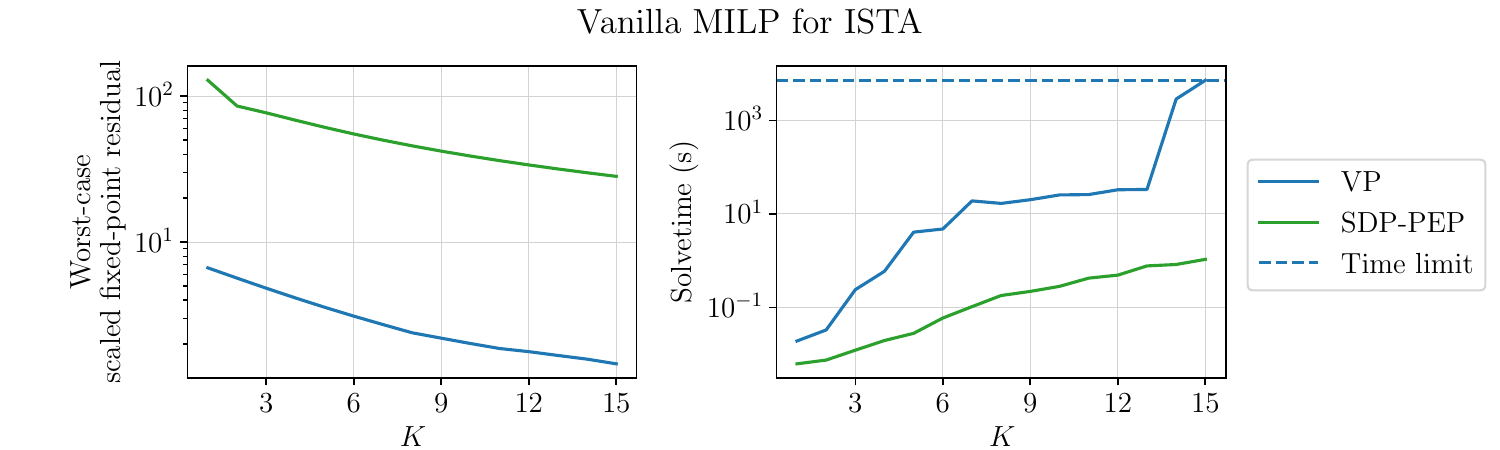}
    \caption{Left: Worst-case fixed-point residual norm for ISTA applied to~\eqref{prob:lasso} via VP and PEP. Right: Computation time to solve the respective optimization problems. 
    }
    \label{fig:vanilla_milp_results}
\end{figure}
Figure~\ref{fig:vanilla_milp_results} shows that our formulation offers almost two orders of magnitude reduction in worst-case fixed-point residual, compared to PEP. However, without any scalability improvement (see next section), we are unable to solve such problems beyond $K=15$ within the two hour time limit.

\section{Scalability improvements}\label{sec:scalability}
In this section, we introduce scalability improvements that allow us to solve~\eqref{eq:vp} for larger values of~$K$.
Modern MILP solvers implement a branch-and-bound (B\&B) algorithm~\cite{scip,gurobi,cplex20} with cutting planes, \ie, branch-and-cut.
The B\&B algorithm systematically partitions the feasible region into sub-regions and obtains bounds via the solution of the associated LP relaxations.
By tracking the best known primal and dual bounds, the B\&B algorithm can prune out suboptimal regions of the feasible space.
The practical efficiency of B\&B solvers is heavily reliant on the tightness of the LP relaxations. 
These can be strengthened by tightening bounds on the continuous variables, such as the big-$M$ constraints introduced in Section~\ref{subsec:piecewise_affine_steps}, or by adding valid linear inequalities, \ie, cutting planes.

We now introduce various improvements that include a combination of bound tightening techniques using standard interval propagation, optimization-based bound tightening, and operator-theory bounds specific to our setting.
Furthermore, we make use of the convex hull inequalities from Section~\ref{subsec:piecewise_affine_steps} to tighten the relaxations at the root node of the B\&B procedure.
Finally, we describe how we solve~\eqref{eq:vp} \emph{sequentially} from iteration $1$ to $K$, by iteratively exploiting the solutions at previous iterations. 

\subsection{Bound tightening}\label{subsec:bound_tightening}
\paragraph{Interval propagation}
Consider an affine step of the form~\eqref{eq:affinestep}.
Using the factorization of $M$, we can convert the affine step to an explicit update $w = \tilde{B} y$ where $\tilde{B} = M^{-1} B$.
Given lower and upper bounds $\ubar{y}, \bar{y}$, we can derive lower and upper bounds on $w$ \cite[Sec.\ 3]{gowalScalableVerifiedTraining2019} via $\ubar{w} = (1/2)( \tilde{B}(\bar{y} + \ubar{y}) - |\tilde{B}|(\bar{y} - \ubar{y}) )$, and $\bar{w} =  (1/2)( \tilde{B}(\bar{y} + \ubar{y}) + |\tilde{B}|(\bar{y} - \ubar{y}))$, where $|\tilde{B}|$ is the elementwise absolute value of $\tilde{B}$.
\errata{For the piecewise affine steps considered above, first bound the affine input $\theta=a^Ty$ by $M^-\le\theta\le M^+$ using the same interval formulas.
The scalar operators $\psi=\ReLU$, $\mathcal T_\lambda$, and $\mathcal S_{[b,c]}$ are nondecreasing, so their outputs satisfy $\psi(M^-)\le w\le\psi(M^+)$.}
Chaining bounds for all iterations~\eqref{eq:iterations} gives interval propagation bounds $\ubar{s}_{\rm ip}^K \le s^K \le \bar{s}_{\rm ip}^{K}$. 

\paragraph{Operator theory based bounds}
For a convergent operator $T$, the fixed-point residual is bounded by a decreasing sequence $\{\alpha_k\}_{k=1}^{\infty}$, \ie, 
$\| s^{k} - s^{k-1} \|_\infty \leq \alpha_k.$
\begin{proposition}\label{prop:fp_bound_prop}
    Let $\|s^{k} - s^{k-1} \|_\infty \leq \alpha_k$ and $\ubar{s}^{k-1} \le s^{k-1} \le \bar{s}^{k-1}$. Then, $ \ubar{s}^{k-1} - \alpha_k \ones \leq s^{k} \leq \bar{s}^{k-1} + \alpha_k \ones$.
\end{proposition}
\begin{proof}
    Since $\|s^{k} - s^{k-1} \|_\infty \leq \alpha_k$, we have $|s_i^{k} - s_i^{k-1}| \leq \alpha_k$. 
    The absolute value implies $s_i^{k} - s_i^{k-1} \leq \alpha_k$ and $s_i^{k-1} - s_i^{k} \leq \alpha_k$.
    Replacing $s_i^{k-1}$ by its upper and lower bounds, respectively, completes the proof.
\end{proof}
In theoretical analysis of first-order methods, the bounding sequence $\alpha_k$ depends on the worst-case initial distance to optimality that we denote by
\begin{equation}\label{prob:miqp_r2}
    \begin{array}{lll}
        R^2 = &\text{maximize} & \| s^0 - s^\star \|^2_2 \\
        &\text{subject to} &  x \in X,~s^\star \in S^\star(x),~\errata{s^0\in S},
    \end{array}
\end{equation}
where $S^\star(x)$ is the set of (primal-dual) optimal solutions for instance $x$.
\errata{We use operator theory bounds only when the initial radius has a finite certified upper bound and the rate constants hold uniformly over the parameter and initial sets.}
Consider a $\beta$-contractive operator $T(\cdot, x)$ for any $x \in X$ with $\beta < 1$, which corresponds to a \emph{linearly} convergent method.
The fixed-point residual is bounded by 
$\alpha_k = 2\beta^{k-1}R$~\cite[Sec.\ B.3]{sambharyaLearningWarmStartFixedPoint2024}.
Consider now an averaged operator $T(\cdot, x)$ for any $x \in X$, which corresponds to a sublinearly convergent method.
In this case, the residual is bounded by $\alpha_k = (D/k^q)R$ for some $q > 0$ and $D$ is a constant independent of $K$~\cite{ryuLargeScaleConvexOptimization2022}.
Using the fact that any vector $\ell_2$-norm is bounded above by its $\ell_1$-norm, we formulate a MILP analogous to that of problem~\eqref{eq:vp} using techniques from Section~\ref{sec:milp_formulation} to compute
\begin{equation}\label{prob:initrad_milp}
    \begin{array}{lll}
        \errata{\bar R} = &\text{maximize} & \| s^0 - s^\star \|_1 \\
        &\text{subject to} &  x \in X,~s^\star \in S^\star(x),~\errata{s^0\in S},
    \end{array}
\end{equation}
over variables $s^\star \in \reals^d$, $x \in \reals^p$, \errata{and $s^0\in\reals^d$}.
\errata{Since $R\le\bar R$, a certified upper bound on $\bar R$ can replace $R$ in the residual bounds.}
These results give the operator theory based bounds $\ubar{s}_{\rm ot}^k = \ubar{s}^{k-1} - \alpha_k \ones \le s^{k} \le \bar{s}_{\rm ot}^{k} = \bar{s}^{k-1} + \alpha_k \ones$.

We apply both interval propagation techniques and operator theory bounds and use the tightest elementwise, \ie, set the bounds as 
\begin{equation}\label{eq:tightest_preprocess_bound}
    \ubar{s}^k = \max\{\ubar{s}^k_{\text{ip}}, \ubar{s}^k_{\text{ot}}\}, \quad \bar{s}^k = \min\{\bar{s}^k_{\text{ip}}, \bar{s}^k_{\text{ot}}\}.
\end{equation}

\paragraph{Optimization-based bound tightening (OBBT)}
In practice, we tighten bounds~$\ubar{s}^K$ and $\bar{s}^K$ obtained in~\eqref{eq:tightest_preprocess_bound} with optimization-based bound tightening \cite{Gleixner2017},
by solving the following LP for each component of $s^K$
\begin{equation}\label{prob:optboundtightLP}
	\!\!\!\!\!\begin{array}{ll}
		\text{min(max)imize} & s^{K}_i\\
		\text{subject to} & (s^{k+1}, s^k)\in \conv(\{s^{k+1} = T(s^k, x)\}),\quad k = 0,\dots,K-1\\
            & s^{k} \in [\ubar{s}^k, \bar{s}^k],\quad k=0,\dots,K.
	\end{array}
\end{equation}
OBBT plays a crucial role in tightening the bounds of the $y$ variables in the definition of $M^+$ and $M^-$ in ~\eqref{eq:bigM} and leads to a tighter LP relaxation.
Without OBBT, we observe the big-$M$ values are large and the MILPs do not scale to larger $K$.
\subsection{Warm-starting between consecutive iterations}
In order for the B\&B algorithm to efficiently prune out suboptimal regions of the feasible space, it needs to exploit quality primal heuristics, which are schemes that aim to quickly provide a feasible solution~\cite{bertholdPrimalHeuristicsInteger2025}\cite[Chapter 9.2]{confortiIntegerProgramming2014}\cite[Section 3.1]{scavuzzoMachineLearningAugmented2024}.
In our setting, after solving the MILP at iteration $K$, we can extract the worst-case instance sequence of iterates $\{s^k\}_{k=0}^K$ and parameter $x$ that achieves the objective value.
We find a quality heuristic solution for the next iterate $K+1$ by applying one extra step of operator $T(\cdot, x)$ and inferring the binary variable values from it.
In practice, we found this procedure allowed the MILP to quickly find the optimal lower bound in the B\&B procedure.
\subsection{Cutting planes at the root node}\label{subsec:cutting_planes}
The theorems from Section~\ref{subsec:piecewise_affine_steps} provide exponentially many inequalities to exactly describe the convex hulls.
However, Proposition~\ref{prop:separation_procedure} gives an efficient procedure by which we separate from the hulls, giving a cutting plane.
We tighten the root node continuous relaxation with these cutting planes by running the separation procedure for each output component of the piecewise affine steps.
In our implementation, we add only the most violated cutting plane, if any exists, to avoid overly increasing the size of the LP relaxations, which may slow down the B\&B execution~\cite{scavuzzoMachineLearningAugmented2024}.

\section{Numerical experiments}\label{sec:experiments}
Our goal is to compare to the best known non-asymptotic convergence bounds.
State-of-the-art computer-assisted analysis techniques, such as the PEP framework, give the best known results of this type for a wide range of function classes.
These are guaranteed to match the known asymptotic rates, and in some cases even improve their constants \cite{bousselmiInterpolationConditionsLinear2024,guTightSublinearConvergence2020,upadhyayaAutomatedTightLyapunov2025}.

We compare four baselines:
\begin{itemize}
    \item \textbf{Sample maximum (SM).}
    We sample $N=500$ problems indexed from the parameter set by $\{x^i\}_{i=1}^{N} \in X$. For each sample, we then run the first-order method for $k = 1,\dots, K$ iterations and compute the maximum $\ell_{\infty}$-norm of the fixed-point residual.
    The sample maximum is a lower bound on the objective value of~\eqref{eq:vp}.    
    \item \textbf{Performance Estimation Problem SDP (SDP-PEP).} We compute the worst-case objective from the PEP framework, which requires an upper bound $R$ to the initial distance to optimality.
    We compute~$R$ as in equation~\eqref{prob:miqp_r2}, by using the squared $\ell_2$-norm in the objective and obtaining a mixed-integer QP.
    When forming the PEP problems, we use the methods based on operating splitting techniques, quadratic function classes, and linear operators~\cite{ryuOperatorSplittingPerformance2020, bousselmiInterpolationConditionsLinear2024}.
    \item \textbf{$\ell_\infty$ scaled PEP (Scaled SDP-PEP).} 
As discussed in Section~\ref{sec:pep_comparison}, PEP is not built to work with the $\ell_\infty$ norm due to its dimension independence\footnote{SDP-PEP often has an optimal solution with a rank-1 Gram matrix (for example,~\cite[Section 3.6]{taylorSmoothStronglyConvex2017}), which yields a 1-dimensional worst-case function. In that case, the $\ell_\infty$ and $\ell_2$ norms of the performance measure are equal. However, such behavior does not hold in general.}.
To make the comparisons fair, we also report the worst-case SDP-PEP objective scaled by $1/\sqrt{d}$, using the fact that for any $v \in \reals^d$, $\|v\|_\infty \ge 1/\sqrt{d} \|v\|_2$.
    \item \textbf{Verification Problem (VP).} We solve the MILP for the VP within a 5\% optimality gap with a 2-hour time limit.
    We solve each MILP sequentially (see Section~\ref{sec:scalability}), and report the worst-case fixed-point residual norm.
\end{itemize}

For each experiment, we provide the scaling matrix $H$ that allows the scaled fixed-point residual to correspond to the KKT conditions of the underlying problem, with proofs in Appendix~\ref{apx:Hmatrix_proofs}.
For comparison, we also provide references to theoretical asymptotic convergence rates that we use to construct our operator theory bounds.
We also compute the solution times for solving the VP and PEP.
To observe the effects of~\eqref{eq:tightest_preprocess_bound}, for each $K$ we also count the fraction of iterate components where the operator theory bound is tighter than bound propagation, with plots found in Appendix~\ref{apx:op_theory_improvements}.
%
All examples are written in Python 3.13.
We solve the MILPs using Gurobi 13.0~\cite{gurobi} with duality gap tolerance $5\%$ and 20 CPU cores.
For PEP, we use the PEPit toolbox~\cite{goujaudPEPitComputerassistedWorstcase2024},  interfaced to Clarabel 0.11.1~\cite{goulartClarabelInteriorpointSolver2024}.

\ifdoubleblind\else
The code for all experiments can be found at
\begin{center}
{\tt \url{https://github.com/stellatogrp/mip_algo_verify}}.
\end{center}
\fi
\subsection{Sparse coding}\label{subsec:full_sparse_coding_experiment}
We use the sparse coding settings from Example~\ref{ex:vanillamilp}.

\paragraph{Algorithms}
We analyze both the ISTA shown by~\eqref{eq:istastep} as well as the fast iterative shrinkage thresholding algorithm (FISTA)~\cite{beckFastIterativeShrinkageThresholding2009}.
The FISTA iterations for problem~\eqref{prob:lasso} require an auxiliary sequence of iterates $w^k$ with 
$w^0 = u^0$ and a precomputed scalar sequence $\beta_{k+1} = (1 + \sqrt{1 + 4 \beta_k^2}) / 2$, with $\beta_0 = 1$~\cite{beckFastIterativeShrinkageThresholding2009}, namely
    \begin{equation}\label{eq:fista_iter}
        \begin{array}{ll}
            u^{k+1} &= \mathcal{T}_{\lambda \eta}((I-\eta D^TD)w^k + \eta D^T x) \\
            w^{k+1} &= u^{k+1} + (\beta_k - 1)/\beta_{k+1} (u^{k+1} - u^k).
        \end{array}
    \end{equation}

    \paragraph{Performance metric}
    The optimality condition of problem~\eqref{prob:lasso} can be written in terms of the subdifferential $\partial \lambda \|\cdot\|_1$ as $0 \in \left( \partial \lambda \|\cdot\|_1 \right) (u) + D^T (D u - x)$.
    The right-hand side with $u = u^K$ of FISTA~\eqref{eq:fista_iter}, for $K \ge 2$, \errata{contains}
    $H(s^K - s^{K-1})$ with $s^k = (u^k, u^{k-1})$
    and $\errata{H = (D^TD - \eta^{-1}I)}\begin{bmatrix}
        I & -(\beta_{K-2} - 1)/(\beta_{K-1}) I 
    \end{bmatrix}$.
    Taking $\beta_k \equiv 1$ recovers the corresponding expression for ISTA~\eqref{eq:istastep} for all $K$.

\paragraph{Problem setup}
We compare algorithms on both strongly convex and nonstrongly convex examples.
For the strongly convex example, we choose $p=20,~n=15$ and for the nonstrongly convex example, we choose $p=15,~n=20$.
We use the same data generating procedure described in Example~\ref{ex:vanillamilp} and choose $\eta= 1/L$, where $L$ is the maximum eigenvalue of $D^TD$.
For the operator theory bounds, we apply results from \cite[Theorem 2.2, Theorem 4.2]{kimAnotherLookFast2018}.
These theorems show that the convergence rates for both ISTA and FISTA in terms of fixed-point residual decrease at rate $O(K^{-1})$.

\paragraph{Results}
Residual results are available in Figure~\ref{fig:lasso_full_resids}, timing results are available in Figure~\ref{fig:lasso_full_times} in Appendix~\ref{apx:experiment_times}, and operator theory bound improvements are available in Appendix~\ref{apx:op_theory_improvements}.
Our approach VP, in all cases, obtains two orders of magnitude reduction in the worst-case fixed-point residual norm over SDP-PEP.
\begin{figure}[t!]
    \centering
    \includegraphics[width=0.8\textwidth]{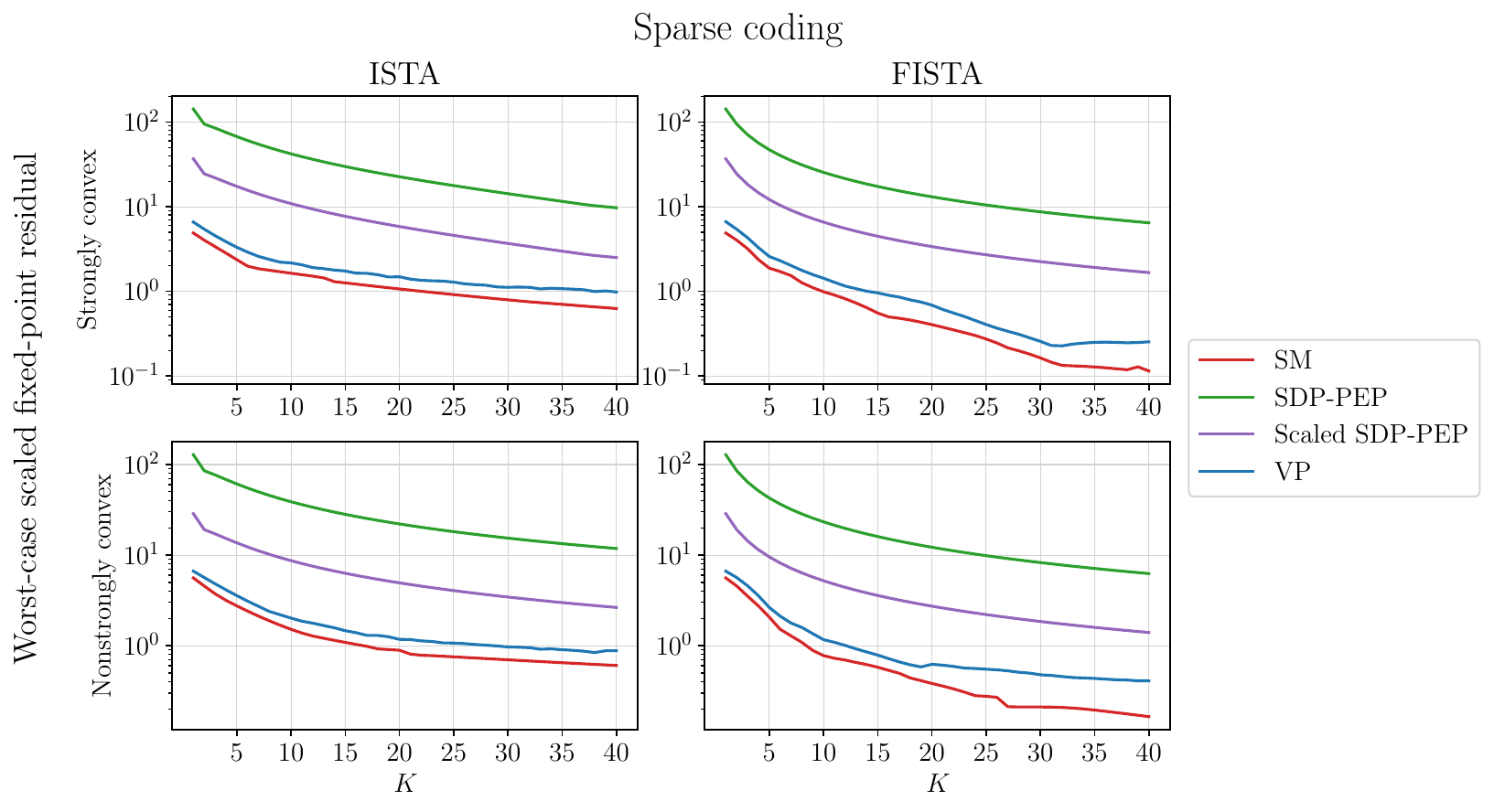}
    \caption{Worst-case fixed point residual norms for the ISTA (first column) and FISTA (second column) applied to~\eqref{prob:lasso}. In both the strongly convex case (first row, $p=20, n=15$) and the nonstrongly convex case (second row, $p=15, n=20$), the acceleration reduces the worst-case residual, with a larger gain in the nonstrongly convex case at $K=40$.}
    \label{fig:lasso_full_resids}
\end{figure}
To show the effect of our scalability improvements,
in Figure~\ref{fig:vanilla_improved_milp_time_comp}, we compare the solve times of the vanilla MILP formulation in Section~\ref{sec:milp_formulation} to the MILP with our improvements: those improvements allow the VP to scale up to $K=40$ within our time limit.
\begin{figure}
    \centering
        \includegraphics[width=0.55\textwidth]{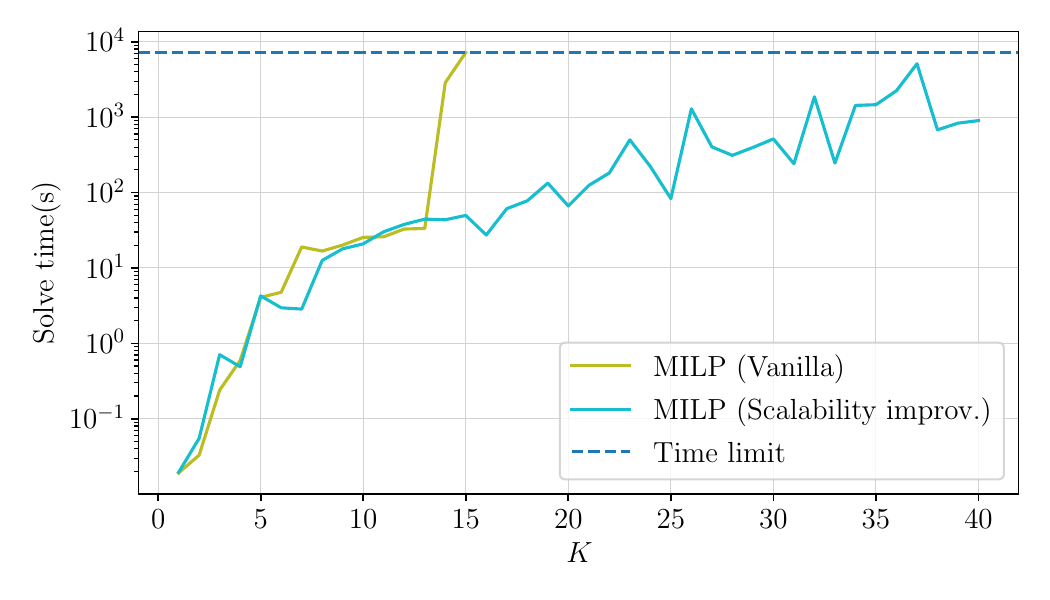}
    \caption{Solve times for the vanilla MILP defined in Section~\ref{sec:milp_formulation} vs.\ the MILP with our scalability improvements in Section~\ref{sec:scalability} on the same nonstrongly convex instance. }    \label{fig:vanilla_improved_milp_time_comp}
\end{figure}
%
\subsection{Network optimization}\label{subsec:flow_experiment}
We consider a minimum cost network flow problem~\cite{bertsekas1991linear} with varying node demands.
Given a set of $n_s$ supply nodes and $n_d$ demand nodes, consider a bipartite graph on the node groups with capacitated directed edges from supply to demand nodes.
Let $A_s, A_d$ be submatrices of rows corresponding to supply nodes and demand nodes, respectively.
For each edge $i$, let $c_i$ be the flow capacity, $\mu_i$ be the cost of sending a unit of flow, and $f_i$ be the variable representing the amount of flow sent.
Lastly, let $b_s \geq 0$ be the vector of supply amounts available at each supply node and $x \leq 0$ be the amount of demand requested at each demand node.
The minimum cost network flow problem becomes
\begin{equation}\label{prob:mincostflow}
	\begin{array}{ll}
		\text{minimize} & \mu^T f\\
		\text{subject to} 
            & A_s f \le b_s,\quad A_d f = x,\quad 0 \leq f \leq c,\\
	\end{array}
\end{equation}
\ie~problem~\eqref{eq:prob} with 
$P = 0 \in \reals^{n_e\times n_e}$,
$A = [-\!I\; I \; A^T_s \; A^T_d]^T \in \reals^{(n_s + n_d + 2n_e) \times n_e }$,
~$q(x) = \mu \in \reals^{n_e}$, $b(x) = (0, c, b_s, x) \in \reals^{n_s + n_d + 2n_e}$, and $C = \reals_+^{n_s + 2n_e} \times \{0\}^{n_d}$.
%

\paragraph{Algorithms}
We solve the problem using the primal-dual hybrid gradient (PDHG) algorithm ~\cite{chambolleFirstOrderPrimalDualAlgorithm2011} that operates on a primal and dual iterate~$s^k = (z^k, v^k, w^k)$.
Note that $(v^k, w^k)$ are the separate components of the dual vector that correspond to the primal inequality and equality constraints, respectively.
We consider both the standard PDHG~\cite{chambolleFirstOrderPrimalDualAlgorithm2011} with fixed step size $\eta > 0$ adapted to LPs~\cite[Equation (3)]{applegatePracticalLargeScaleLinear2021}, 
\begin{equation}\label{eq:pdhg_vanilla}
	\begin{array}{ll}
		z^{k+1} & = \mathcal{S}_{[0, c]}(z^k - \eta (\mu + A_s^T v^k - A_d^Tw^k))\\
            v^{k+1} &= \max\{v^k + \eta(-b_s + A_s(2z^{k+1} - z^k)), 0\}\\
		w^{k+1} & = w^k + \eta(x - A_d(2z^{k+1} - z^k)).
    \end{array}
\end{equation}
We also consider a version with momentum in the primal variable:
\begin{equation}\label{eq:pdhg_momentum}
    \begin{array}{ll}
		z^{k+1} & = \mathcal{S}_{[0, c]}(z^k - \eta (\mu + A_s^T v^k - A_d^Tw^k))\\
            \tilde{z}^{k+1} &= z^{k+1} + \xi_k(z^{k+1} - z^k)\\
            v^{k+1} &= \max\{v^k + \eta(-b_s + A_s(2\tilde{z}^{k+1} - z^k)), 0\}\\
		w^{k+1} & = w^k + \eta(x - A_d(2\tilde{z}^{k+1} - z^k)),
    \end{array}
\end{equation}
where we determine the momentum coefficients with Nesterov's update rule $\xi_k = k/(k+3)$~\cite{nesterov1983method}.
For both algorithms, we choose the step size as $\eta = 0.5/\left\| [ -\!A^T_s \; A^T_d] \right\|_2$.
%

    \paragraph{Performance metric}
    The optimality condition of problem~\eqref{prob:mincostflow} is given by
    \begin{equation*}
        0 \in \begin{bmatrix}
            \mu + A_s^T v - A_d^T w + \normalcone_{[0, c]} (z) \\
            - (A_s z - b_s) + \normalcone_{\ge 0} (v) \\
            A_d z - x
        \end{bmatrix},
    \end{equation*}
    where $\normalcone_{[0, c]}$ and $\normalcone_{\ge0}$ refer to the normal cone of sets $[0, c]$ and nonnegative orthant~$\reals_+^{m_1}$, respectively.
    The right-hand side of the optimality condition at $s = s^K$ for $s^k = (z^k, v^k, w^k)$ of momentum PDHG~\eqref{eq:pdhg_momentum} \errata{contains} $H (s^K - s^{K-1})$ with
    \begin{equation*}
        H = - \begin{bmatrix}
            \eta^{-1} I & - A_s^T & A_d^T \\
            - (1 + 2 \xi_{K-1}) A_s & \eta^{-1} I & 0 \\
            (1 + 2 \xi_{K-1}) A_d & 0 & \eta^{-1} I
        \end{bmatrix}.
    \end{equation*}
    We recover a similar result for standard PDHG~\eqref{eq:pdhg_vanilla} when $\xi_k \equiv 0$.

\paragraph{Problem setup}
We generate a random bipartite graph on 15 supply nodes and 10 demand nodes by adding each edge with probability 0.5.
After generating the random graph and forming the LP, we obtain $A \in \reals^{m \times n}$ with $n_s=15, n_d=10, n_e=63$.
For this experiment, we keep a fixed supply $b_s = 10$ for each node and a fixed capacity $c = 5$ for each edge, and we sample the costs i.i.d.\ from a uniform distribution $\mu_i \sim \mathcal{U}[5, 10]$.
We parameterize only the demand as a unit hypercube $X = [-7, -6]^{n_d}$.
We choose $s^0$ to be the primal-dual solution to the problem instance where all demand nodes have maximum demand.
If this instance has a solution, then all instances in the family are feasible.
For the operator theory bounds, we apply~\cite[Theorem 1]{luGeometryRefinedRate2024}, which says that the fixed-point residual for both algorithms converges at rate $O(K^{-1/2})$.

\paragraph{Results}
Residual results are available in Figure~\ref{fig:pdhg_resids} and timing results are available in Figure~\ref{fig:pdhg_times} in Appendix~\ref{apx:experiment_times}.
In both cases,~\eqref{eq:vp} is able to capture that the worst-case fixed-point residual norm is not monotonically decreasing.
Our approach is able to show residual bounds for this LP family up to $K=70$ and benefits of momentum.
\begin{figure}[h]
    \centering
\includegraphics[width=0.8\textwidth]{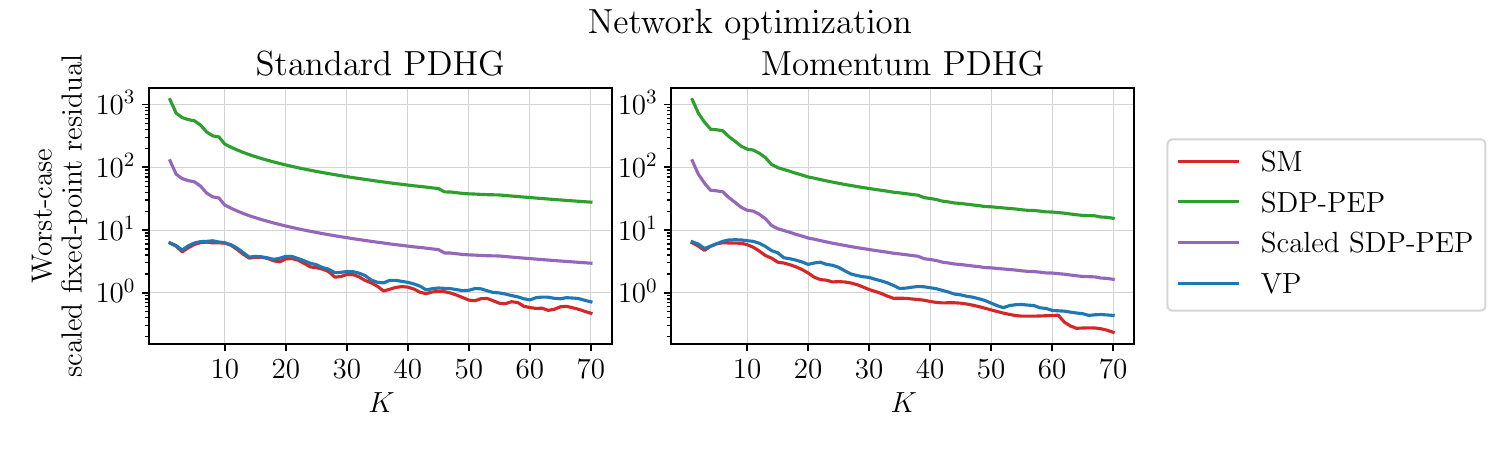}
    \caption{Worst-case fixed-point residual norms for the standard PDHG algorithm (left) and the version with momentum (right). Both algorithms show a rippling behavior and the momentum algorithm obtains a lower residual at $K=70$. 
    }
    \label{fig:pdhg_resids}
\end{figure}
\subsection{Optimal control}\label{subsec:control_experiment}
We consider the problem of controlling a constrained time-invariant linear dynamical system with $n_s$ states and $n_u$ inputs, over a finite time horizon indexed as $t = 1, \dots, T$~\cite[Chapter 11.3]{borrelliPredictiveControlLinear2017}.
Given dynamics matrices $A^{\text{dyn}} \in \reals^{n_s \times n_s}$ and $B^{\text{dyn}} \in \reals^{n_s \times n_u}$, along with positive semidefinite state cost \errata{matrices $Q,Q_T \in \symm_{+}^{n_s \times n_s}$}, positive definite input cost matrix~$R \in \symm_{++}^{n_u \times n_u}$, we aim to solve
\begin{equation}\label{prob:optimalcontrol}
	\begin{array}{ll}
		\text{minimize} & s_T^TQ_Ts_T + \sum_{t=1}^{T-1} s_t^T Qs_t + u_t^T R u_t\\
		\text{subject to} 
            & s_{t+1} = A^\text{dyn} s_t + B^\text{dyn} u_t, \quad t=1,\dots,T-1\\
            & s_\text{min} \leq s_t \leq s_\text{max}, \quad t=1,\dots,T\\
            & u_\text{min} \leq u_t \leq u_\text{max}, \quad t=1,\dots,T-1\\
            & s_1 = s^{\rm init},
	\end{array}\!\!\!\!\!
\end{equation}
where $s^{\rm init} \in \reals^{n_s}$ is the initial state. 
By defining a new variable $z = (u_1, \dots u_{T-1}) \in \reals^{(T-1)n_u}$, problem~\eqref{prob:optimalcontrol} can be written as function of inputs only~\cite[Section 11.3]{borrelliPredictiveControlLinear2017}, 
\begin{equation}\label{prob:oc_reformed}
	\begin{array}{ll}
		\text{minimize} & (1/2) z^T P z + q(x)^T z\\
		\text{subject to} 
            & l(x) \leq Mz \leq u(x),
	\end{array}
\end{equation}
where the parameter $x = s^{\rm init}$, \errata{$m_c=Tn_s+(T-1)n_u$}, and $l(x),\;u(x) \in \errata{\reals^{m_c}}$, $P \in \symm_{++}^{(T-1)n_u \times (T-1)n_u}$, and $M \in \errata{\reals^{m_c \times (T-1)n_u}}$.
Problem~\eqref{prob:oc_reformed} can be written in the form~\eqref{eq:prob}, with \errata{$A = [-M^T \; M^T]^T \in \reals^{2m_c \times (T-1)n_u}$}, $q(x) \in \reals^{(T-1)n_u}$, \errata{$b(x) = (-l(x), u(x)) \in \reals^{2m_c}$}, and \errata{$C = \reals^{2m_c}_+$}.

\paragraph{Algorithms}
We solve problem~\eqref{prob:oc_reformed} using the alternating direction method of multipliers~(ADMM) from the OSQP solver \cite{ banjacInfeasibilityDetectionAlternating2019,stellatoOSQPOperatorSplitting2020} given $\rho, \sigma > 0$ and primal-dual iterate $(z^0, v^0)$:
\begin{equation}\label{eq:admm_fixedpt}
    \begin{array}{l}
        w^{k+1} = \mathcal{S}_{[l(x), u(x)]}(v^k) \\
        \text{Solve } (P + \sigma I + \rho M^T M)z^{k+1} = \sigma z^k - q(x) + \rho M^T(2w^{k+1} - v^k)\\
        v^{k+1} = v^k + Mz^{k+1} - w^{k+1}.
    \end{array}
\end{equation}

    \paragraph{Performance metric}
    By defining $y^k = \rho (v^{k-1} - w^k)$ and by construction of algorithm~\eqref{eq:admm_fixedpt}, the pair $(w^k, y^k)$ can be shown to satisfy the relevant normal cone conditions of~\eqref{prob:optimalcontrol} at every iteration~\cite[Section 4.1]{banjacInfeasibilityDetectionAlternating2019}.
    So, the only optimality conditions of problem~\eqref{prob:oc_reformed} that we need to measure are primal and dual feasibility:
    \begin{equation*}
        0 = \begin{bmatrix}
            P z + M^T y + q(x) \\
            M z - w
        \end{bmatrix}.
    \end{equation*}
   The right-hand side at $s = s^K$ for $s^k = (z^k, v^k)$ is equivalent to $H (s^K - s^{K-1})$ with
    \begin{equation*}
        \errata{H = \begin{bmatrix}
            - \sigma I & - \rho M^T \\
            0 & I
        \end{bmatrix}}.
    \end{equation*}

\paragraph{Problem setup}
We consider the dynamics of a linearized quadcopter from~\cite{kouzoupisProperAssessmentQP2015} and a time horizon $T=5$.
The input and state cost matrices are $Q = Q_T = \diag(0, 0, 10, 10, 10, 10, 0, 0, 0, 5, 5, 5)$, and $R = 0.1 I$.
We consider a setting with only input constraints $u_{\rm min} = (-0.99) \ones$  and $u_{\rm max} = (2.41) \ones$ (\ie, $s_{\rm min} = - \infty \ones$ and $s_{\rm max} = +\infty \ones$).
We allow the initial state $x = s_{\rm init}$ to live in 
$X = \{ x \in \reals^{12} \mid -\pi/6 \le x_1, x_2 \le \pi/6,\; 0 \le x_3 \le 1, x_i = 0,\; i=4,\dots, 12\}$, \ie, hovering positions with altitude between $0$ and $1$.
We compare three versions of algorithm~\eqref{eq:admm_fixedpt} with $\rho=0.1, 1, 10$ and fixed $\sigma=10^{-4}$, and we use operator theory bounds from~\cite[Theorem 2]{giselssonLinearConvergenceMetric2017}.
This bound says that the convergence rate of the fixed-point residual decreases at rate $O(\tau^K)$, where $\tau = (1/2) + (1/2) \max\{(\rho L - 1)/(\rho L + 1), (1 - \rho\mu)/(\rho \mu + 1)\}$, with $\mu, L$, the minimum and maximum eigenvalues of $P$, respectively.

\paragraph{Results}
Residual results are available in Figure~\ref{fig:osqp_resids} and timing results are available in Figure~\ref{fig:mpc_times} in Appendix~\ref{apx:experiment_times}.
Our VP approach is able to quantify worst-case residual differences and be used to tune step size parameters.
\begin{figure}[h]
    \centering
\includegraphics[width=0.8\textwidth]{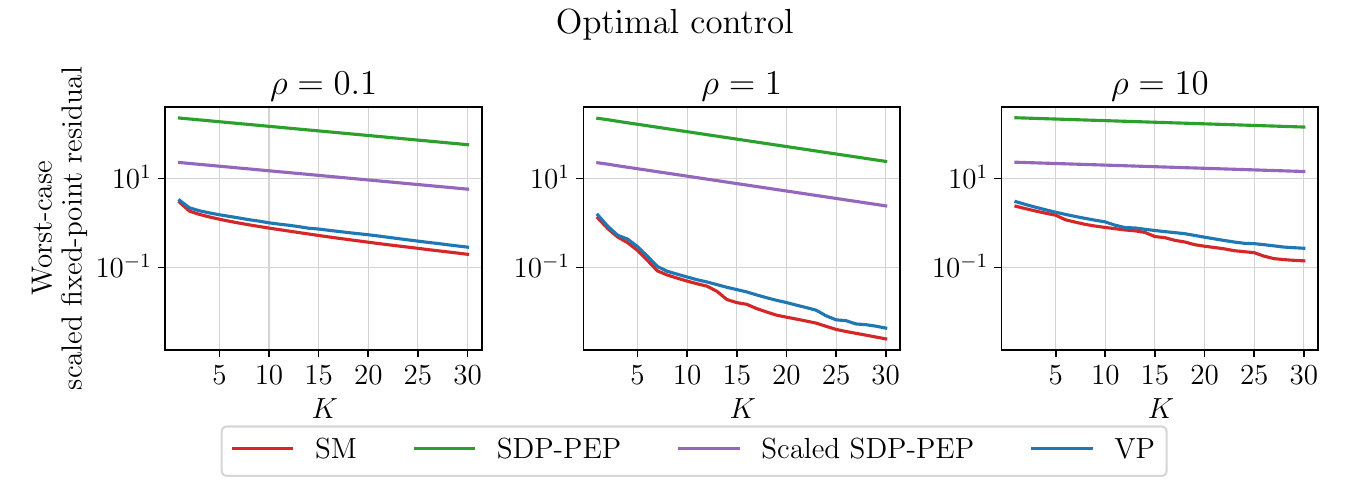}
    \caption{Worst-case fixed-point residual norms for the OSQP algorithm differing step sizes $\rho$. The middle choice $\rho=1$ is able to achieve the lowest worst-case residual by $K=50$. Our VP is also able to show the overall slower convergence with $\rho = 0.1, 10$.}
    \label{fig:osqp_resids}
\end{figure}

\section{Conclusion}
We presented a new MILP framework to exactly verify the performance of first-order optimization algorithms for LPs and QPs.
We formulated the verification problem as an MILP by maximizing the $\ell_\infty$-norm of the convergence residual after a fixed number of iterations over a given parametric family of problems and set of initial iterates.
By constructing convex hull formulations, we represented a variety of piecewise affine steps in addition to directly embedding affine steps into the MILP.
We also incorporated many techniques to enhance the scalability of the MILP, specifically through a custom bound tightening procedure combining interval propagation, operator theory, and optimization-based bound tightening.
Using different examples, we benchmark our method against state-of-the-art computer-assisted performance estimation techniques.
While our analysis was restricted to QPs in this work, we believe it motivates similar studies in parametric convex settings.

\newcommand{\myack}{Vinit Ranjan and Bartolomeo Stellato are supported by the NSF CAREER Award ECCS-2239771. Jisun Park (affiliated with the Research Institute of Mathematics, Seoul National University) is supported by the National Research Foundation of Korea (NRF) grant funded by the Korea government (MSIT) (RS-2024-00353014). Bartolomeo Stellato and Jisun Park are also supported by the ONR YIP Award N000142512147. The authors are pleased to acknowledge that the work reported in this paper was substantially performed using the Princeton Research Computing resources at Princeton University which is a consortium of groups led by the Princeton Institute for Computational Science and Engineering (PICSciE) and Research Computing.

Stefano Gualandi acknowledges the contribution of the National Recovery and Resilience Plan, Mission 4 Component 2 - Investment 1.4 - National Center for HPC, Big Data and Quantum Computing (project code: CN\_00000013) - funded by the European Union - NextGenerationEU.}

\ifpreprint
\section*{Acknowledgements}
\myack
\fi

\bibliography{mip-fom-verify}

\appendix

\section{Full timing results}\label{apx:experiment_times}
In this section, we provide plots of all solve times in seconds for each experiment in Section~\ref{sec:experiments}.
Figure~\ref{fig:lasso_full_times} provides timing results for the experiments in Section~\ref{subsec:full_sparse_coding_experiment}, Figure~\ref{fig:pdhg_times} provides timing results for the experiments in Section~\ref{subsec:flow_experiment}, and Figure~\ref{fig:mpc_times} respectively for the experiments in Section~\ref{subsec:control_experiment}.
\begin{figure}[h]
    \centering
        \includegraphics[width=0.9\textwidth]{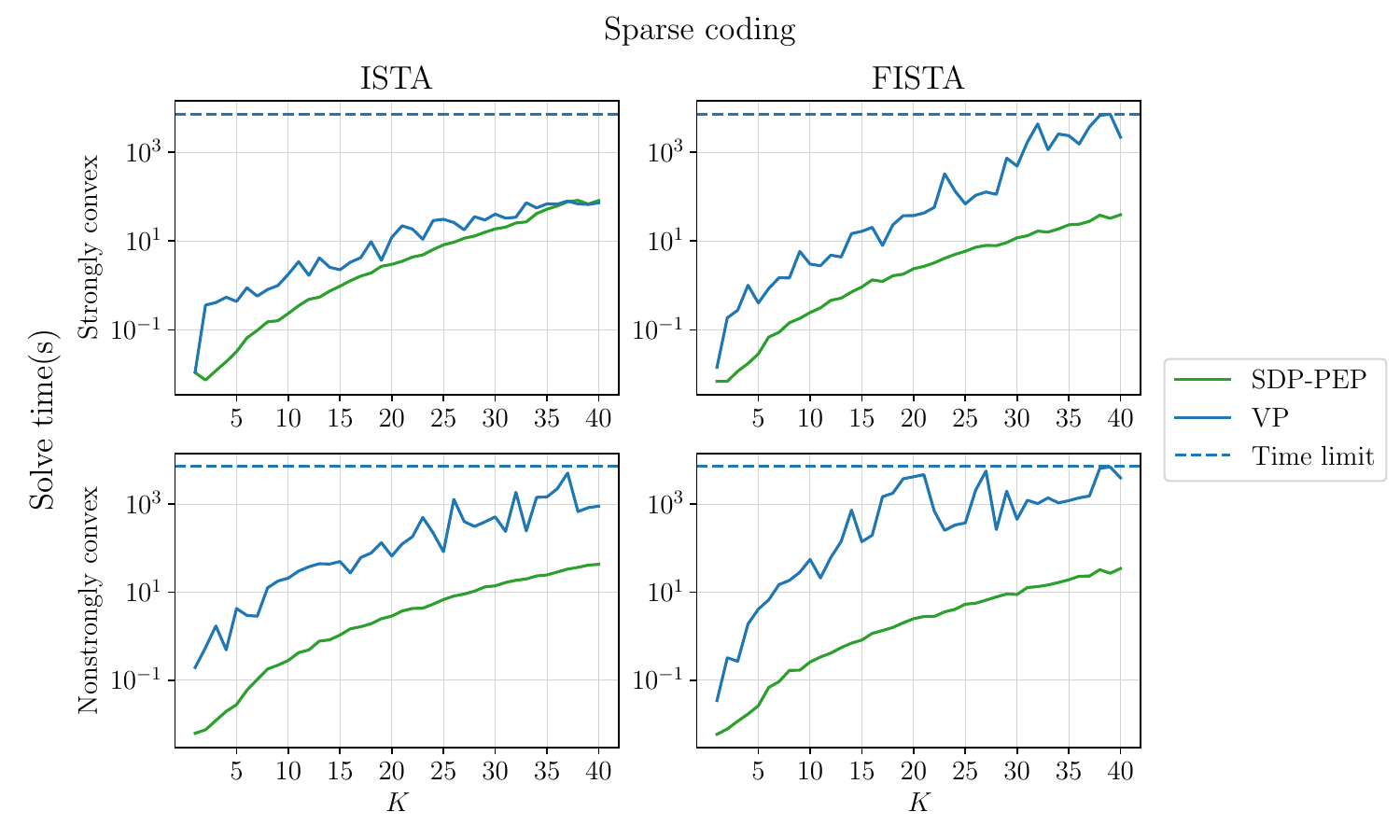}
    \caption{Solve time in seconds for all experiments shown in Figure~\ref{fig:lasso_full_resids}. In both cases, FISTA takes longer to verify than ISTA.}
    \label{fig:lasso_full_times}
\end{figure}
\begin{figure}[h]
    \centering
        \includegraphics[width=0.8\textwidth]{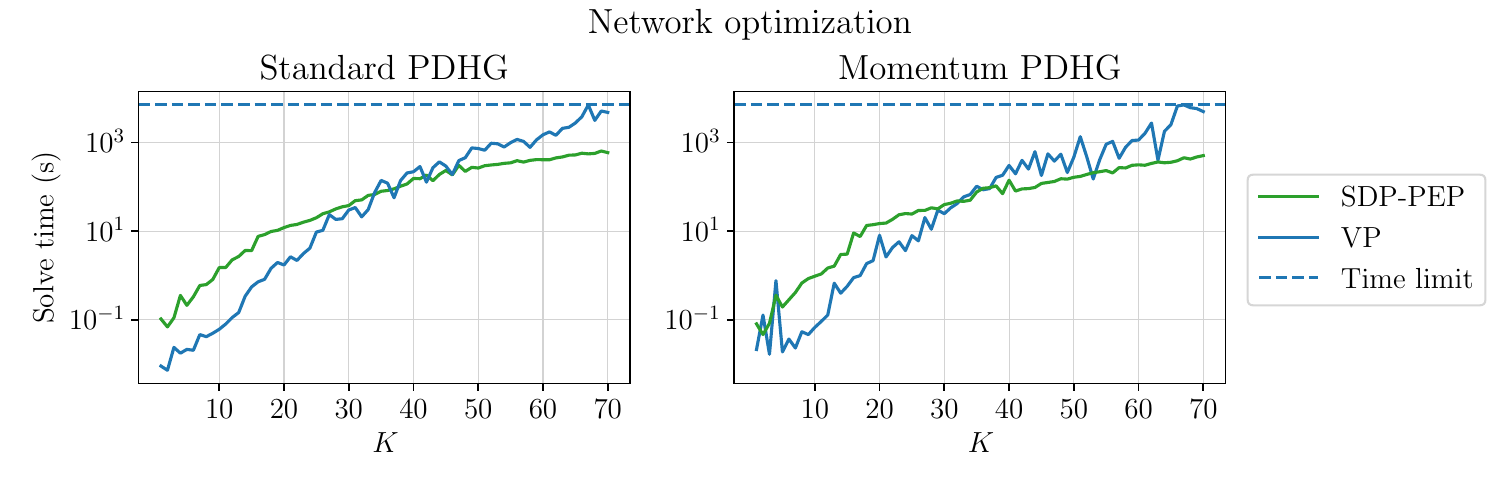}
    \caption{Solve time in seconds for all experiments shown in Figure~\ref{fig:pdhg_resids}. Both versions display similar overall growth in solve time.}
    \label{fig:pdhg_times}
\end{figure}

\begin{figure}[h]
    \centering
        \includegraphics[width=0.8\textwidth]{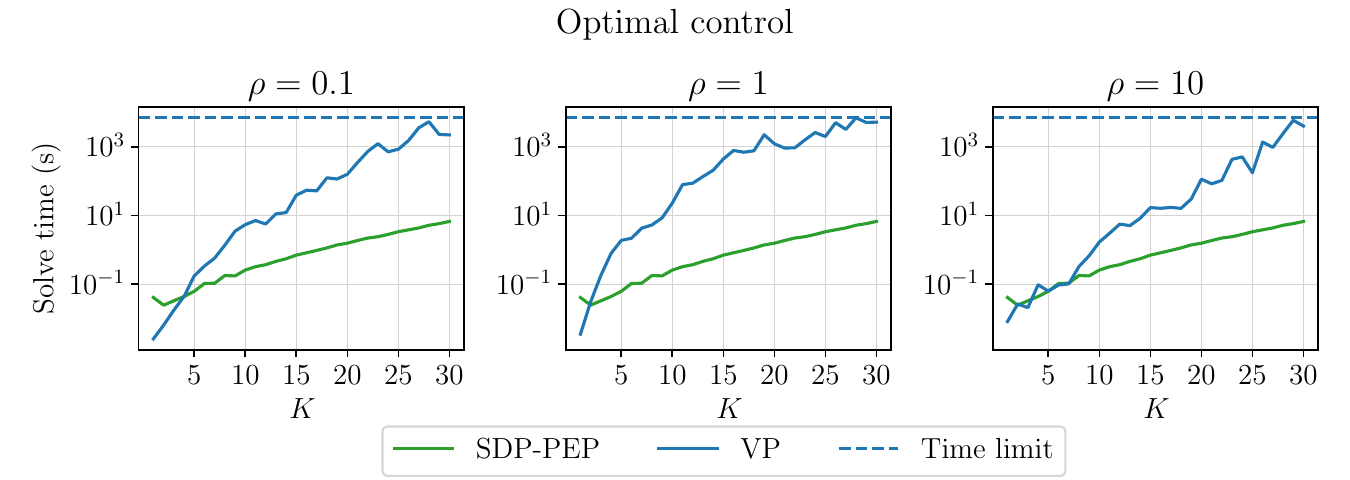}
    \caption{Solve time in seconds for all experiments shown in Figure~\ref{fig:osqp_resids}. All three versions display similar overall growth in solve time.}
    \label{fig:mpc_times}
\end{figure}

\section{Performance metrics of Section~\ref{sec:experiments}}\label{apx:Hmatrix_proofs}
We use the following property of proximal operators.
\begin{lemma}[Proposition 16.44 of \cite{bauschkeConvexAnalysisMonotone2017}]
\label{lem:prox}
    Suppose $g \colon \reals^d \to \errata{\reals\cup\{+\infty\}}$ is a closed, convex, and proper function.
    Then, $ x - p = (I - \prox_{g}) (x) \in \partial g (p)$ for any $x \in \errata{\reals^d}$ and $p = \prox_{g} (x)$.
\end{lemma}
\subsection{Sparse coding}
For notational simplicity, let $g(u) = \lambda \|u\|_1$.
Then, the FISTA update~\eqref{eq:fista_iter} with Lemma~\ref{lem:prox} gives
\begin{equation*}
    (I - \eta D^T D) (w^k - u^{k+1})
    \in \eta \partial g(u^{k+1}) + \eta D^T (D u^{k+1} - x).
\end{equation*}
Combining this with~$u^k - u^{k+1} = (u^k - w^k) + (w^k - u^{k+1})$, we get
\begin{align*}
    &(I - \eta D^T D) (u^k - u^{k+1})
    - \frac{\beta_{k-1} - 1}{\beta_k} (I - \eta D^T D) (u^{k-1} - u^{k}) \\
    &\in
        \eta \partial g(u^{k+1})
        + \eta D^T (D u^{k+1} - x).
\end{align*}
The result for ISTA also directly follows from taking $\beta_k \equiv 1$.

\subsection{Network optimization}
From the momentum PDHG update~\eqref{eq:pdhg_momentum},
\begin{align*}
    &z^{k} - z^{k+1} \\
    &=
        (I - \Pi_{[0, c]}) \big( z^k - \eta (\mu + A_s^T v^k - A_d^Tw^k) \big)
        + \eta (\mu + A_s^T v^k - A_d^T w^k) \\
    &\in
        \normalcone_{[0, c]} (z^{k+1})
        + \eta (\mu + A_s^T v^k - A_d^T w^k) \\
    &=
        \normalcone_{[0, c]} (z^{k+1})
        + \eta (\mu + A_s^T v^{k+1} - A_d^T w^{k+1})
        + \eta A_s^T (v^k - v^{k+1})
        - \eta A_d^T (w^k - w^{k+1}),
\end{align*}
where $\normalcone_{[0, c]}$ refers to the normal cone of $[0, c]$ and the third inclusion holds from Lemma~\ref{lem:prox}.
Similarly, we get
\begin{align*}
    v^k - v^{k+1}
    &\in
        \normalcone_{\ge0} (v^{k+1})
        + \eta \big( b_s - A_s (2\tilde{z}^{k+1} - z^k) \big), \\
    w^k - w^{k+1}
    &= - \eta \big( x - A_d (2 \tilde{z}^{k+1} - z^k) \big),
\end{align*}
where~$\normalcone_{\ge0}$ is the normal cone of the nonnegative orthant.
Using the fact that~$2 \tilde{z}^{k+1} - z^k = z^{k+1} + (1 + 2 \xi_k) (z^{k+1} - z^{k})$, the updates can be written as 
\begin{align*}
    v^k - v^{k+1}
    &\in
        \normalcone_{\ge0} (v^{k+1})
        + \eta \big( b_s - A_s z^{k+1} \big)
        + \eta \left( 1 + 2 \xi_k \right) A_s (z^{k} - z^{k+1}) \\
    w^k - w^{k+1}
    &=
        \eta (A_d z^{k+1} - x)
        - \eta \left(1 + 2 \xi_k \right) A_d (z^k - z^{k+1}).
\end{align*}

\subsection{Optimal control}
With auxiliary variable $y^{k+1} = \rho (v^k - w^{k+1})$, ADMM update~\eqref{eq:admm_fixedpt} can be rewritten as
\begin{align*}
    &w^{k+1} = \Pi_{[l(x), u(x)]}(v^k) \\
    &y^{k+1} = \rho (v^k - w^{k+1}) \\
    &\text{Solve }\,  (P + \sigma I + \rho M^T M)z^{k+1} = \sigma z^k - q(x) + \rho M^T (w^{k+1} - \rho^{-1} y^{k+1}) \\
    &v^{k+1} = v^k + Mz^{k+1} - w^{k+1}.
\end{align*}
Note that the $z^{k+1}$-update is equivalent to
\begin{equation*}
        \sigma (z^k - z^{k+1})
        = (P z^{k+1} + M^T y^{k+1} + q(x)) + \rho M^T (M z^{k+1} - w^{k+1}).
\end{equation*}
Therefore,
\begin{align*}
    \begin{bmatrix}
        z^{k+1} \\ v^{k+1}
    \end{bmatrix}
    - 
    \begin{bmatrix}
        z^k \\ v^k
    \end{bmatrix}
    &=
    \begin{bmatrix}
        - \sigma^{-1} (P z^{k+1} + M^T y^{k+1} + q(x))
            - \sigma^{-1} \rho M^T (M z^{k+1} - w^{k+1}) \\
        (M z^{k+1} - w^{k+1})
    \end{bmatrix} \\
    &=
    - \sigma^{-1} \begin{bmatrix}
        I & \rho M^T \\
        0 & - \sigma I
    \end{bmatrix}
    \begin{bmatrix}
        P z^{k+1} + M^T y^{k+1} + q(x) \\
        M z^{k+1} - w^{k+1}
    \end{bmatrix},
\end{align*}
which is equivalent to
\begin{equation*}
    \begin{bmatrix}
        P z^{k+1} + M^T y^{k+1} + q(x) \\
        M z^{k+1} - w^{k+1}
    \end{bmatrix}
    =
    \begin{bmatrix}
        - \sigma I & - \rho M^T \\
        0 & I
    \end{bmatrix}
    \left(
        \begin{bmatrix}
            z^{k+1} \\ v^{k+1}
        \end{bmatrix}
        - 
        \begin{bmatrix}
            z^k \\ v^k
        \end{bmatrix}
    \right).
\end{equation*}

\section{Operator theory based improvements}\label{apx:op_theory_improvements}
Figures~\ref{fig:lasso_theory_improvements} and~\ref{fig:mpc_theory_improvements} show the fraction of iterate indices where the operator theory based bound provides a tighter bound than interval propagation for their respective experiments.
Note that for the network optimization experiments, the operator theory bounds were never tighter than the interval propagation bounds.
\begin{figure}[t]
    \centering
        \includegraphics[width=0.8\textwidth]{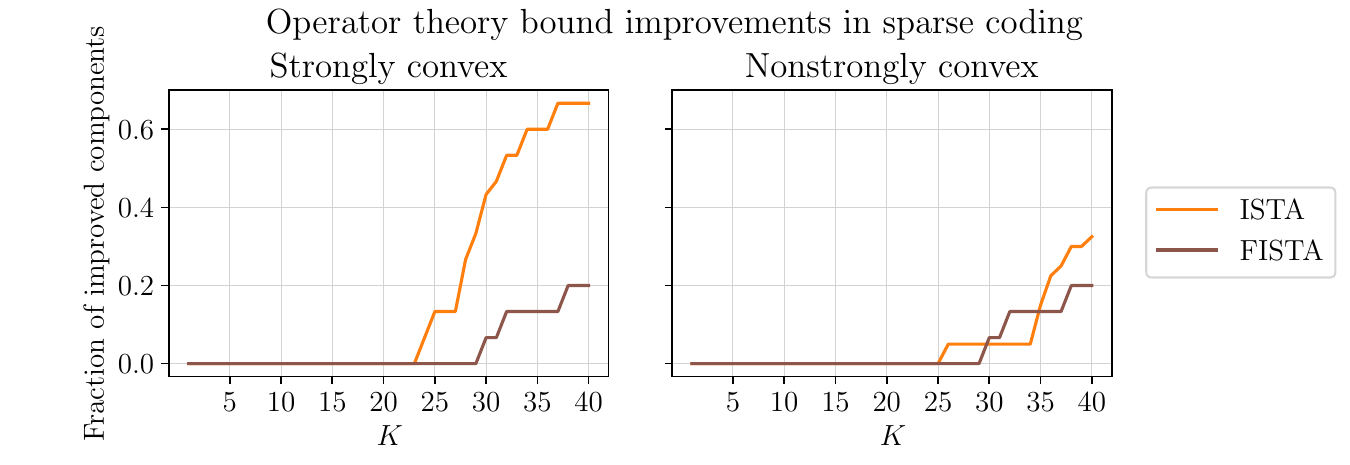}
    \caption{Fraction of iterate components in which the operator theory provides a tighter bound than interval propagation in the strongly convex (left) and nonstrongly convex (right) cases. The fractions correspond to the experiments in Figure~\ref{fig:lasso_full_resids}.
    }
    \label{fig:lasso_theory_improvements}
\end{figure}
\begin{figure}[t]
    \centering
        \includegraphics[width=0.7\textwidth]{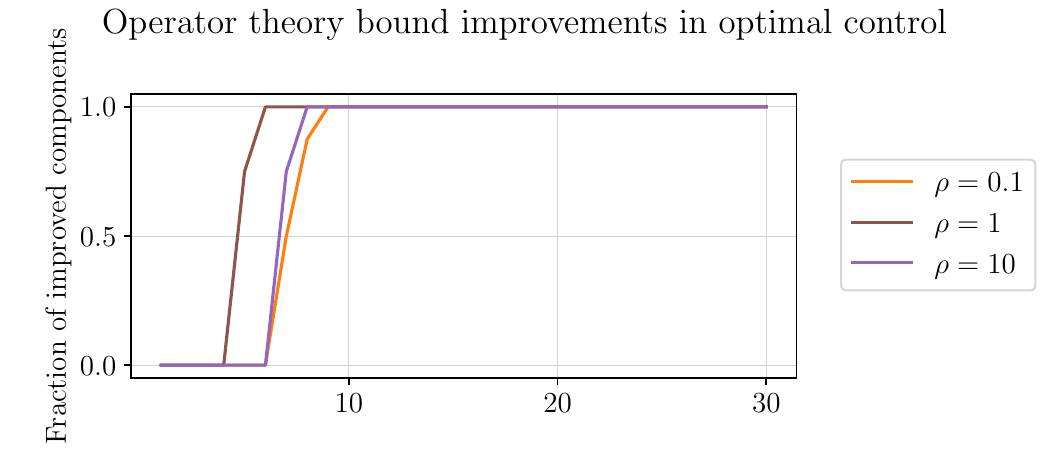}
    \caption{Fraction of iterate components in which the operator theory provides a tighter bound than interval propagation across $\rho$ values. The fractions correspond to the experiments in Figure~\ref{fig:osqp_resids}. The operator theory bounds quickly exploit linear convergence of ADMM.}
    \label{fig:mpc_theory_improvements}
\end{figure}



\end{document}